\documentclass[a4paper]{amsart}
\usepackage[all,2cell]{xy} \UseAllTwocells
\usepackage{enumerate}
\usepackage{amssymb}
\usepackage{color}
\usepackage{fullpage}

\newcommand{\ie}{{\em i.e.\ }}
\newcommand{\eg}{{\em e.g.\ }}
\newcommand{\cf}{{\em cf.\ }}

\newtheorem{theorem}[subsection]{Theorem}

\newtheorem{lemma}[subsection]{Lemma}
\newtheorem{proposition}[subsection]{Proposition}
\newtheorem{corollary}[subsection]{Corollary}

\theoremstyle{definition}

\newtheorem{nonexample}[subsection]{Non-example}
\newtheorem{notation}[subsection]{\textbf{Notation}}

\newtheorem{example}[subsection]{Example}

\newtheorem{remark}[subsection]{\textbf{Remark}}

\numberwithin{equation}{section}

\newcommand{\N}{\mathbf{N}}
\newcommand{\Z}{\mathbf{Z}}

\newcommand{\ko}{\: , \;}

\newcommand{\ul}{\underline}
\newcommand{\we}{\wedge}

\renewcommand{\tilde}[1]{\widetilde{#1}}

\newcommand{\ra}{\rightarrow}

\newcommand{\arr}[1]{\stackrel{#1}{\rightarrow}}

\newcommand{\opname}[1]{\operatorname{\mathsf{#1}}}

\renewcommand{\mod}{\opname{mod}\nolimits}
\newcommand{\Mod}{\opname{Mod}\nolimits}

\newcommand{\Sum}{{\opname{Sum}}}

\newcommand{\id}{\mathbf{1}}

\renewcommand{\L}{\mathbf{L}}
\newcommand{\ten}{\otimes}

\newcommand{\colim}{\opname{colim}\nolimits}

\newcommand{\Mcolim}{\opname{Mcolim}}

\newcommand{\op}[1]{\opname{#1}\nolimits}

\renewcommand{\H}[1]{{H}^{#1}}

\newcommand{\ca}{{\mathcal A}}

\newcommand{\cc}{{\mathcal C}}
\newcommand{\cd}{{\mathcal D}}

\newcommand{\ch}{{\mathcal H}}
\newcommand{\ci}{{\mathcal I}}

\newcommand{\cs}{{\mathcal S}}
\newcommand{\ct}{{\mathcal T}}
\newcommand{\cu}{{\mathcal U}}

\newcommand{\cy}{{\mathcal Y}}

\newcommand{\eps}{\varepsilon}

\renewcommand{\phi}{\varphi}

\newcommand{\Hom}{\opname{Hom}}
\newcommand{\RHom}{\opname{RHom}}
\newcommand{\REnd}{\opname{REnd}}

\newcommand{\cone}{\opname{Cone}\nolimits}

\newcommand{\per}{\opname{per}}

\newcommand{\comment}[1]{}

\makeindex
\begin{document}
\title[Weight structures and simple dg modules]{Weight structures and simple dg modules for positive
dg algebras}

\author{Bernhard Keller}
\address{Bernhard Keller\\
Universit\'e Paris Diderot -- Paris 7\\
UFR de Math\'ematiques\\
Institut de Math\'ematiques de Jussieu, UMR 7586 du CNRS \\
Case 7012\\
B\^{a}timent Chevaleret\\
75205 Paris Cedex 13\\
France } \email{keller@math.jussieu.fr}

\author{Pedro Nicol\'as}
\address{Pedro Nicol\'as, Universidad de Murcia, Facultad de Educaci\'on, Campus de Espinardo, 30100 Murcia, ESPA\~NA}
\email{pedronz@um.es}

\thanks{The second named author has done part of this work while being a postdoctoral Fellow of the Fundaci\'on S\'eneca de la Regi\'on de Murcia}

\date{\today}


\begin{abstract} Using techniques due to Dwyer--Greenlees--Iyengar we construct weight structures in triangulated categories generated by compact objects. We apply our result to show that, for a dg category whose homology vanishes in negative degrees and is semi-simple in degree $0$, each simple module over the homology lifts to a dg module which is unique up to isomorphism in the derived category. This allows us, in certain situations, to deduce the existence of a canonical $t$-structure on the perfect derived category of a dg algebra. From this, we can obtain a bijection between hearts of $t$-structures and sets of so-called simple-minded objects for some dg algebras (including Ginzburg algebras associated to quivers with potentials). In three appendices, we elucidate the relation between Milnor colimits and homotopy colimits and clarify the construction of $t$-structures from sets of compact objects in triangulated categories as well as the construction of a canonical weight structure on the unbonded derived category of a non positive dg category.
\end{abstract}


\maketitle
\tableofcontents

\section{Introduction}
Finite-dimensional modules over an associative unital algebra may be
described as built up from simple modules or as presented by
projective modules. The interplay between these two descriptions is
at the heart of the interpretation  of Koszul duality for dg
algebras (and categories) given in \cite{Keller1994a}, cf. also
\cite{DwyerGreenlees02} \cite{DwyerGreenleesIyengar06}. However, in
order to apply this theory a dg algebra $A$, we need the `simple dg
$A$-modules' as an additional datum. Clearly, a necessary condition
for the existence of such dg modules is that the homology $H^*(A)$
should be equipped with a suitable set of graded simple modules. One
may ask whether this condition is also sufficient. Now realizing
modules over the homology $H^*(A)$ as homologies of dg modules over
$A$ is in general a hard problem, \cf for example
\cite{BensonKrauseSchwede04}. In this paper, we treat one class of
dg algebras where the problem of realizing the simple homology
modules has a satisfactory solution. We define this class by merely
imposing conditions on the homology $H^*(A)$: It should be
concentrated in degrees $\geq 0$ and semi-simple in degree~$0$. Let
us point out that if $H^0(A)$ is a field, our result follows from
Propositions~3.3 and 3.9 of \cite{DwyerGreenleesIyengar06}, as
kindly explained to us by Srikanth Iyengar \cite{Iyengar10}. The
class we consider contains the Koszul duals of smooth dg algebras
$B$ whose homology is concentrated in non positive degrees and
finite-dimensional in each degree. Important examples of these are
the Ginzburg dg algebras associated to quivers with potential
\cite{Ginzburg} \cite{KellerYang}. The proof of our result is based
on the construction of canonical weight structures on suitable
triangulated categories (section~\ref{s:weight-structures}) in
analogy to results obtained by
Pauksztello (Theorem~2.4 of
\cite{Pauksztello08}). These weight structures are also useful in a
second application, namely the construction of a $t$-structure on
the perfect derived category of a dg algebra $A$ in our class
(section~\ref{s:construction-of-t-structures}). This $t$-structure
has as its left aisle the closure under extensions, positive shifts
and direct summands of the free module $A$. Its heart is a length
category whose simple objects are the indecomposable factors of $A$
in $\per A$. Let us point out that the existence of this
$t$-structure also follows from a recent result by Rickard-Rouquier
\cite{RickardRouquier10}.

As another application, we establish a bijection between families of
`simple-minded objects' (a piece of terminology due to J.~Rickard and,
indepently, to K\"onig-Liu \cite{KoenigLiu})
and hearts of $t$-structures in suitable triangulated categories
(section~\ref{s:hearts}). Further applications will be given in the
forthcoming paper \cite{KellerNicolas11}.

In establishing our main theorem, we need foundational results on
the precise link between Milnor colimits and homotopy colimits (in
the sense of derivators) and on the construction of $t$-structures
from sets of compact objects. We prove these in the two appendices.
In another appendix, we prove the existence of a canonical weight structure on
the (unbounded) derived category of a non positive dg category,
in analogy with a result by Bondarko \cite[\S 6]{Bondarko0704.4003v8}.

\section{Acknowledgments}
The authors thank Chris Brav and David Pauksztello for stimulating
conversations on the material of this paper. They are very grateful
to Srikanth Iyengar for pointing out reference
\cite{DwyerGreenleesIyengar06} and explaining how Propositions~3.3
and 3.9 in that paper imply Corollary~\ref{lifting graded simples}
below in the case where $H^0A$ is a field. They are also very
grateful to Dong Yang for carefully reading the first version of
this article.

\section{Terminology and notations}\label{notation}

In this article, `graded' will always mean `$\Z$-graded', and `small' will be
frequently used to mean `set-indexed'. A {\em length category} is an abelian category
where each object has finite length.

We write $\Sigma$ for the shift functor of any triangulated category. Let
\[L_{0}\arr{f_{0}}L_{1}\arr{f_{1}}L_{2}\ra\dots
\]
be a sequence of morphisms in a triangulated category $\cd$.  Its
{\em Milnor colimit} \cite{Milnor1962} is an object, denoted by $\op{Mcolim}_{n\geq
0}L_{n}$, which fits into the {\em Milnor triangle},
\[\xymatrix{
\coprod_{n\geq 0}L_{n}\ar[r]^{\id-\sigma} & \coprod_{n\geq 0}L_{n}\ar[r] & \op{Mcolim}_{n\geq 0}L_{n}\ar[r] & \Sigma\coprod_{n\geq 0}L_{n},
}
\]
where $\sigma$ is the morphism with components
\[
\xymatrix{
L_{n}\ar[r]^{f_{n}} & L_{n+1}\ar[r]^{\text{can}\hspace{0.5cm}} & \coprod_{n\geq 0}L_{n}.
}
\]
Thus, the Milnor colimit is determined up to a (non unique) isomorphism.
The notion of
Milnor colimit has appeared in the literature under the name of
\emph{homotopy colimit} (see \cite[Definition
2.1]{BokstedtNeeman1993}, \cite[Definition 1.6.4]{Neeman2001}).
However, Milnor colimits are not functorial and, in general, they do
not take a sequence of triangles to a triangle of $\cd$. Thus, we
think it is better to keep this terminology for the notions
appearing in the theory of derivators \cite{Maltsiniotis2001,
Maltsiniotis2007, CisinskiNeeman2005}. For a study of the
relationship between Milnor colimits and homotopy colimits see our
Appendix 1.

Let
\[\dots\ra L_{2}\arr{f_{1}}L_{1}\arr{f_{0}}L_{0}
\]
be a sequence of morphism in a triangulated category $\cd$. Its {\em Milnor limit} is an object, denoted by $\op{Mlim}_{n\geq
0}L_{n}$, which fits into the triangle,
\[\xymatrix{
\Sigma^{-1}\prod_{n\geq 0}L_{n}\ar[r] & \op{Mlim}_{n\geq 0}L_{n}\ar[r] &\prod_{n\geq 0}L_{n}\ar[r]^{\id-\sigma} & \prod_{n\geq 0}L_{n},
}
\]
where $\sigma$ is the morphism with components
\[
\xymatrix{
L_{n}\ar[r]^{f_{n-1}} & L_{n-1}\ar[r]^{\text{can}\hspace{0.5cm}} & \prod_{n\geq 0}L_{n}
}
\]
for $n\neq 0$, and the zero map in the component $0$
\[0: L_{0}\ra\prod_{n\geq 0}L_{n}.
\]
As in the case of the Milnor colimit, the Milnor limit is determined up to a (non unique) isomorphism.

If $\cd$ is a triangulated category and $\cs$ is a set of objects of
$\cd$, we denote by $\op{thick}_{\cd}(\cs)$ the smallest full
subcategory of $\cd$ containing $\cs$ and closed under extensions,
shifts and direct summands.

Let $k$ be a field and $A$ a dg $k$-algebra. We denote the derived
category of $A$ by $\cd A$, \cf \cite{Keller1994a}. The {\em perfect
derived category} of $A$, denoted by $\per A$, is
$\op{thick}_{\cd A}(A)$. The {\em finite-dimensional derived
category} of $A$, denoted by $\cd_{fd}A$, is the full subcategory of
$\cd A$ formed by those dg modules $M$ whose homology is of finite
total dimension:
\[\sum_{p\in\Z}\op{dim}_{k}\H{p}(M)<\infty.
\]

\section{Weight structures from compact objects}
\label{s:weight-structures}
\label{s:weight-structures-from-compact-objects}

Let us recall
the definition of a weight structure from \cite{Bondarko0704.4003v8} and
\cite{Pauksztello08}  (it is called co-$t$-structure in \cite{Pauksztello08}):
A {\em weight structure} on a triangulated category $\ct$ is
a pair of full additive subcategories $\ct^{>0}$ and $\ct^{\leq 0}$ of $\ct$
such that
\begin{itemize}
\item[w0)] both $\ct^{>0}$ and $\ct^{\leq 0}$ are stable under taking direct
factors;
\item[w1)] the subcategory $\ct^{>0}$ is stable under $\Sigma^{-1}$ and
the subcategory $\ct^{\leq 0}$ is stable under $\Sigma$;
\item[w2)] we have $\ct(X, Y)=0$ for all $X$ in $\ct^{>0}$ and
all $Y$ in  $\ct^{\leq 0}$;
\item[w3)] for each object $X$ of $\ct$, there is a {\em truncation triangle}
\[
\sigma_{>0}(X) \to X \to \sigma_{\leq 0}(X) \to \Sigma \sigma_{>0}(X)
\]
with $\sigma_{>0}(X)$ in $\ct^{>0}$ and $\sigma_{\leq 0}(X)$ in $\ct^{\leq 0}$.
\end{itemize}
Notice that the objects $\sigma_{>0}(X)$ and $\sigma_{\leq 0}(X)$ in
the truncation triangle are not functorial in $X$. The following
theorem and its proof are based on Propositions~3.3 and 3.9 of
\cite{DwyerGreenleesIyengar06}. Compared to the main result of
\cite{Pauksztello10}, the theorem has stronger hypotheses: assumption c)
is not present in [loc. cit.]; but it also has a stronger conclusion:
the description of the weight structure in terms of homology.

\begin{theorem}\label{canonical weight structure}
Suppose that $\ct$ is a triangulated category with small coproducts and that $\cs\subset\ct$ is a full additive subcategory stable under taking direct summands such that
\begin{itemize}
\item[a)] $\cs$ {\em compactly generates} $\ct$, \ie the functors $\ct(S,?):\ct\ra\Mod\Z\ko S\in\cs$, commute with small coproducts, and if $M\in\ct$ satisfies $\ct(\Sigma^pS,M)=0$ for all $p\in\Z\ko S\in\cs$, then $M=0$;
\item[b)] we have $\ct(L, \Sigma^p M)=0$ for all $L$ and $M$ in $\cs$ and all integers $p<0$;
\item[c)] the category $\Mod\cs$ of additive functors $\cs^{op}\to \Mod\Z$ is semi-simple.
\end{itemize}
For $X$ in $\ct$ and $p\in\Z$, we write $H^pX$ for the object $L \mapsto \ct(L,\Sigma^p X)$ of $\Mod\cs$. Then we have:
\begin{itemize}
\item[1)] There is a unique weight structure $(\ct^{>0}, \ct^{\leq 0})$ on $\ct$ such that $\ct^{\leq 0}$ is formed by the objects $X$ with $H^pX=0$ for all  $p>0$ and $\ct^{>0}$ is formed by the objects $X$ with $H^pX=0$ for all $p\leq 0$.
\item[2)] For each object $X$, there is a truncation
triangle
\begin{equation}
\label{eq:truncation-triangle}
\sigma_{>0}(X) \to X \to \sigma_{\leq 0}(X) \to \Sigma \sigma_{>0}(X)
\end{equation}
such that the morphism $X \to \sigma_{\leq 0}(X)$ induces an isomorphism in $H^p$ for $p\leq 0$ and the morphism $\sigma_{>0}(X) \to X$ induces an isomorphism in $H^p$ for $p>0$.
\end{itemize}
\end{theorem}
\begin{proof}
Let $\Sum(\cs)$ be the closure under small coproducts of $\cs$ in $\ct$.

{\em 1st step: The functor $H^0: \Sum(\cs) \to \Mod\cs$ is an equivalence.} Indeed, this functor is fully faithful because the objects of $\cs$ are compact in $\ct$. It is an equivalence because $\Mod\cs$ is semi-simple and $\cs$ stable under direct factors.


{\em 2nd step: For each object $X$ of $\ct$ and each integer $m$, there is a morphism $V_m(X) \to X$ such that $V_m(X)$ belongs to $\Sigma^{-m} \Sum(\cs)$
and the induced map
\[
H^m(V_m(X)) \to H^mX
\]
is an isomorphism.} Indeed, by the first step, the module $H^mX$ is isomorphic to $H^m( V_m(X))$ for some $V_m(X)$ lying in $\Sigma^{-m} \Sum(\cs)$.

{\em 3rd step: For each object $X$ of $\ct$, there is a triangle
\[
V(X) \to X \to C(X) \to \Sigma V(X)
\]
such that $V(X)$ is a sum of objects $\Sigma^p L$, where $L\in\cs$ and $p<0$,
the map
\[
H^pX \to H^p(C(X))
\]
is bijective for all $p\leq 0$, and the map
\[
H^p(V(X)) \to H^pX
\]
is surjective for all $p>1$.} Indeed, we define
\[
V(X) = \coprod_{m>0} V_m(X)
\]
and $C(X)$ to be the cone over the natural morphism $V(X) \to X$. Then we obtain the claim because the functors $H^p$ commute with coproducts,
the $H^pM$ vanishes for all $M\in\cs$ and all $p<0$ and the morphism $V(X) \to X$ induces an isomorphism
\[
H^1(V(X)) \to H^1X.
\]

{\em 4th step: For each object $X$ of $\ct$, there is a triangle
\[
\sigma_{>0}(X) \to X \to \sigma_{\leq 0} X \to \Sigma \sigma_{>0}(X)
\]
such that $\sigma_{>0}(X)$ lies in $\ct^{>0}$ and  $\sigma_{\leq 0} X$
lies in $\ct^{\leq 0}$.}
We iterate the construction of the third step to obtain a direct system
\[
X \to C(X) \to C^2(X) \to \cdots \to C^p(X) \to \cdots
\]
and define
\[
\sigma_{\leq 0}(X) = \op{Mcolim} C^p(X).
\]
We define $\sigma_{>0}(X)$ by the above triangle. The compactness of the objects of $\cs$ in $\ct$ implies that each functor $H^n$, $n\in\Z$, takes
Milnor colimits to colimits in $\Mod\cs$. Let us show that $\sigma_{\leq 0}(X)$ belongs to $\ct^{\leq 0}$. Indeed, for $n>0$, by construction, the morphisms
$C^p(X) \to C^{p+1}(X)$ induce the zero map in $H^n$. Thus, the module
\[
H^n(\sigma_{\leq 0}(X)) = H^n(\op{Mcolim} C^p(X)) = \colim H^n(C^p(X))
\]
vanishes for $n>0$. Let us show that $\sigma_{>0}(X)$ belongs to $\ct^{>0}$. Indeed,
by induction on $p$, we see that the object $K_p(X)$ defined by the triangle
\[
K_p(X) \to X \to C^p(X) \to \Sigma K_p(X)
\]
belongs to  $\ct^{>0}$. By considering the exact sequence
\[
H^{n-1}X \to H^{n-1}(C^p(X)) \to H^n(K_p(X)) \to H^nX \to H^n(C^p(X))
\]
we see that for each $n\leq 0$, the morphism
\[
H^{n-1}X \to H^{n-1}(C^p(X))
\]
is surjective and the morphism
\[
H^nX \to H^n(C^p(X))
\]
is injective. By passing to the colimit over $p$, we
obtain that for each $n\leq 0$, the morphism
\[
H^{n-1}X \to H^{n-1}(\sigma_{\geq 0}(X))
\]
is surjective and the morphism
\[
H^nX \to H^n(\sigma_{\geq 0}(X))
\]
is injective. By the exact sequence
\[
H^{n-1}X \to H^{n-1}(\sigma_{\geq 0}(X)) \to H^n(\sigma_{>0}(X)) \to H^nX \to H^n(\sigma_{\geq 0}(X))
\]
associated with the truncation triangle, this implies that
for each $n\leq 0$, the module $H^n(\sigma_{>0}(X))$ vanishes.

{\em 5th step: For each object $X$ of $\ct$ and each $n\leq 0$, the map $H^nX \to H^n(\sigma_{\leq 0}(X))$ is an isomorphism and for $n>0$, the map $H^n(\sigma_{>0}(X)) \to H^nX$ is an
isomorphism.} Indeed, the first claim follows from the fact that
$X \to C^p(X)$ induces an isomorphism in $H^n$ for all $n\leq 0$,
which we obtain by induction from the third step. For the second claim,
we consider the exact sequence
\[
H^{n-1}X \to H^{n-1}(\sigma_{\leq 0}(X)) \to H^n(\sigma_{>0}(X)) \to H^nX \to H^n(\sigma_{\leq 0}(X)).
\]
For $n=1$, the first map is an isomorphism and the last term
vanishes; for $n\geq 2$, the second and the last term vanish.

{\em 6th step: If $X$ is an object of $\ct$ and $Y$ an object of $\ct^{\leq 0}$,
each morphism $X\to Y$ factors through $X \to \sigma_{\leq 0}(X)$.} Indeed,
since $V(X)$ is a coproduct of objects $\Sigma^{-m}L$, $m>0$, $L\in\cs$,
by the triangle
\[
V(X) \to X \to C(X) \to \Sigma V(X)\ko
\]
the given morphism factors through $C(X)$. By induction, one constructs
a compatible system of factorizations
\[
\xymatrix{ X \ar[r] & C^p(X) \ar[r]^-{f_p} & Y.}
\]
This system lifts to a factorization $X \to \op{Mcolim}(C^p(X)) \to Y$, which
proves the claim since $\sigma_{\geq 0}(X) = \op{Mcolim}(C^p(X))$.

{\em 7th step: For $X\in \ct^{>0}$ and $Y\in \ct_{\leq 0}$, we have
$\ct(X,Y)=0$.} Indeed, let $f: X \to Y$ be a morphism. By the
6th step, it factors through $X \to \sigma_{\leq 0}(X)$.
We claim that $Z=\sigma_{\leq 0}(X)$ vanishes. Indeed, by
the 4th step, we have $H^nZ=0$ for $n>0$ and by the 5th
step, we have $H^nZ=0$ for $n\leq 0$ since $H^nX$ vanishes
for $n\leq 0$.

{\em 8th step: the conclusion.} Axioms w0) and w1) are clear, axiom
w2) has been shown in the 7th step and axiom w3) in the 4th step.
Claim b) has been shown in the 5th step.
\end{proof}

Although the assignment $X\mapsto\sigma_{\leq 0}X$ in part~2) of
Theorem~\ref{canonical weight structure} is not uniquely defined up
to isomorphism and it is not functorial, we have the following
useful result:

\begin{lemma}\label{good properties of sigma}
In the situation of Theorem~\ref{canonical weight structure}, we have:
\begin{itemize}
\item[1)] $\sigma_{\leq 0}(X\oplus Y)=\sigma_{\leq 0}(X)\oplus\sigma_{\leq 0}(Y)$,
\item[2)] $\sigma_{\leq 0}(\Sigma^{p}X)=\Sigma^p\sigma_{\leq 0}(X)$.
\end{itemize}
\end{lemma}

\section{Positive dg algebras}

\begin{corollary}\label{canonical weight structure for positive dg algebra}
Let $k$ be a commutative associative ring with unit. Let $\ca$ be a small dg $k$-linear category such that:
\begin{itemize}
\item[a)] $H^p\ca$ vanishes for $p<0$,
\item[b)] $\Mod H^0(\ca)$ is a semisimple abelian category.
\end{itemize}
Then we have:
\begin{itemize}
\item[1)] There exists a weight structure $w=((\cd\ca)^{w>0},(\cd\ca)^{w\leq 0})$ on $\cd\ca$ such that $(\cd\ca)^{w>0}$ is formed by those modules $X$ such that $H^pX=0$ for $p\leq 0$ and $(\cd\ca)^{w\leq 0}$ is formed by those modules $X$ such that $H^pX=0$ for $p>0$.
\item[2)] For each module $X$ there exists a truncation triangle
\[\sigma_{>0}(X)\ra X\ra\sigma_{\leq 0}(X)\ra\Sigma\sigma_{>0}(X)
\]
such that the morphism $X\ra \sigma_{\leq 0}(X)$ induces an isomorphism in $H^p$ for $p\leq 0$ and the morphism $\sigma_{>0}(X)\ra X$ induces
an isomorphism in $H^p$ for $p>0$.
\end{itemize}
\end{corollary}
\begin{proof}
We apply Theorem~\ref{canonical weight structure} by taking $\ct=\cd\ca$ and $\cs$ to be the full subcategory of $\cd\ca$ formed by the direct summands of finite direct sums of modules of the form $A^\we=\ca(?,A)$ where $A$ is an object of $\ca$. Thanks to \cite{Keller1994a} we know that $\cd$ is compactly generated by $\cs$ and that condition~a) implies $\Hom_{\cd\ca}(L,\Sigma^pM)=0$ for all $L$ and $M$ in $\cs$ and all integers $p<0$. After restricting scalars along the functor $H^0\ca\ra\cs$ we get an equivalence
\[\Mod H^0(\ca)\arr{\sim}\Mod \cs.
\]
Thus, condition~b) implies that $\Mod \cs$ is semisimple.
\end{proof}

\begin{nonexample}
If $H^0A$ is not semisimple we do not have a triangle as the one in
part~2) of Corollary~\ref{canonical weight structure for positive dg
algebra}. We can take, for example, the algebra of dual numbers
$A=k[\varepsilon]$ with $\varepsilon^2=0$ over field $k$ and
consider the complex $M$ equal to the cone over the map $\eps: A \to
A$. Let $S$ be the simple $A$-module. If there was a triangle
\[
\sigma_{\geq 0}(M) \to M \to \sigma_{<0}(M) \to \Sigma \sigma_{\geq
0}(M),
\]
the object $\sigma_{\geq 0}(M)$ would have to be isomorphic to $S$
and the object $\sigma_{<0}(M)$ to $\Sigma S$ (because the homology
of $M$ is concentrated in degrees $0$ and $-1$ and isomorphic to $S$
in both degrees). Then the connecting morphism
\[
\sigma_{<0}(M) \to \Sigma \sigma_{\geq 0}(M)
\]
would be a morphism $\Sigma S \to \Sigma S$ and thus would have to
be $0$ or an isomorphism. In the first case, we find that $M$ is
decomposable, a contradiction, and in the second case, we find that
$M$ is a zero object, a contradiction as well.
\end{nonexample}

\begin{notation}\label{notation canonical weight structure}
In analogy with the case of $t$-structures, we say that the weight structure of the Corollary~\ref{canonical weight structure for positive dg algebra} is the {\em canonical weight structure}. If $\ca$ is in fact a dg algebra $A$, we write $S_{A}=\sigma_{\leq 0}A$.
\end{notation}

\begin{lemma}\label{H0 fully faithful}
Let $A$ be an arbitrary dg algebra. If $M\in\cd A$ and $P$ is a direct summand of a small coproduct of copies of $A$, then the morphism of $k$-modules induced by $\H0$
\[\Hom_{\cd A}(P,M)\ra\Hom_{\H0A}(\H0P,\H0M)
\]
is an isomorphism.
\end{lemma}
\begin{proof}
The full subcategory of $\cd A$ formed by the objects $P$ satisfying the assertion contains $A$ and is closed under small coproducts and direct summands.
\end{proof}

\begin{lemma}\label{indecomposable decomposition of A}
Let $A$ be a dg algebra such that in $\Mod H^0(A)$, the module
$H^0A$ admits a finite decomposition into indecomposables (\eg
$H^0A$ is semisimple). There exists a decomposition into
indecomposables $A=\bigoplus_{i=1}^{r}A_{i}$ of $A$ in $\cd A$ such
that $\H0A=\bigoplus_{i=1}^{r}\H0(A_{i})$ is a decomposition into
indecomposables of $\H0A$ in $\Mod\H0(A)$.
\end{lemma}
\begin{proof}
A decomposition of $\H0A$ into indecomposables in the category of $\H0A$-modules gives us a complete family $\{e'_{1},\dots, e'_{r}\}$ of primitive orthogonal idempotents of the ring $\op{End}_{\H0A}(\H0A)$.
Now, by using the ring isomorphism
\[\H0:\op{End}_{\cd A}(A)\arr{\sim}\op{End}_{\H0A}(\H0A)
\]
we find a complete family $\{e_{1},\dots, e_{r}\}$ of primitive orthogonal idempotents of the ring $\op{End}_{\cd A}(A)$.  Since idempotents split in $\cd A$, each $e_{i}$ has an image $A_{i}$ in $\cd A$ and we obtain that $A=\bigoplus_{i=1}^{r}A_{i}$ is a decomposition of $A$ into indecomposables in $\cd A$.
\end{proof}

\begin{proposition}\label{the subcategory generated by the simple}
Let $A$ be a dg algebra with homology concentrated in non negative degrees and such that $H^0A$ is a semi-simple ring.
\begin{itemize}
\item[1)] Let $X$ be an object of $\cd A$ with bounded homology and such that each $H^nX\ko n\in\Z$, is a finitely generated $H^0A$-module.
If $p\in\Z$ is an integer such that $H^nX=0$ for $n>p$ and $H^pX\neq 0$, then $X$ belongs to the smallest full subcategory $\op{susp}^{\oplus}(\Sigma^{-p}S_{A})$ of $\cd A$ containing $\Sigma^{-p}S_{A}$ and closed under extensions, positive shifts and direct summands.
\item[2)] Assume that each $H^nA\ko n\in\Z$, is a finitely generated $H^0A$-module. Then if $M\in\per A$, for any truncation triangle
\[\sigma_{>p}(M)\ra M\ra \sigma_{\leq p}(M)\ra \Sigma\sigma_{>p}(M)
\]
we have $\sigma_{\leq p}M\in\op{susp}^{\oplus}(\Sigma^{-p}S_{A})$.
\end{itemize}
\end{proposition}
\begin{proof}
1) We will use induction on the width of the interval delimited by those degrees with non-vanishing homology. By Lemmas~\ref{H0 fully faithful} and \ref{indecomposable decomposition of A}, there are direct summands $A_{1}\ko\dots\ko A_{r}$ of $A$ in $\cd A$, natural numbers $n_{1}\ko\dots\ko n_{r}$, and a morphism $f: \bigoplus_{i=1}^{r}\Sigma^{-p}A_{i}^{n_{i}}\ra X$ in $\cd A$ such that $H^pf$ is an isomorphism in $\Mod H^0A$. Consider truncation triangles
\[\sigma_{>0}(A_{i})\ra A_{i}\ra\sigma_{\leq 0}(A_{i})\ra\Sigma\sigma_{>0}(A_{i}),
\]
as the ones in part~b) of Theorem~\ref{canonical weight structure}. After Lemma~\ref{good properties of sigma} we know that the objects $\sigma_{\leq 0}A_{i}$ can be taken to be direct summands of $S_{A}$ in $\cd A$. In particular, the $\Sigma^{-p}A_{i}$ are objects of $\op{susp}^{\oplus}(S_{A})$. Now notice that $X\in(\cd A)^{w\leq p}$, and so it is right orthogonal to the objects of the wing $(\cd A)^{w>p}$. Hence the morphism $f$ factors through the morphism $\bigoplus_{i=1}^{r}\Sigma^{-p}A_{i}^{n_{i}}\ra\bigoplus_{i=1}^{r}\Sigma^{-p}\sigma_{\leq 0}(A_{i})^{n_{i}}$:
\[\xymatrix{
\bigoplus_{i=1}^{r}\Sigma^{-p}\sigma_{>0}(A_{i})^{n_{i}}\ar[r]\ar[dr]_{0} & \bigoplus_{i=1}^{r}\Sigma^{-p}A_{i}^{n_{i}}\ar[r]\ar[d]^{f} & \bigoplus_{i=1}^{r}\Sigma^{-p}\sigma_{\leq 0}(A_{i})^{n_{i}}\ar[r]\ar@{.>}[dl]^{\tilde{f}} & \bigoplus_{i=1}^{r}\Sigma^{-p+1}\sigma_{>0}(A_{i})^{n_{i}} \\
& X&&
}
\]
Since $H^p(\tilde{f})$ is an isomorphism, for the mapping cone $X'$ of $\tilde{f}$ the width of the interval delimited by those degrees with non-vanishing homology is strictly smaller than that of $X$, and $H^n(X')=0$ for $n>p-1$. By induction hypothesis we get $X'\in\op{susp}^{\oplus}(\Sigma^{-p+1}S_{A})$, which implies that $X\in\op{susp}(\Sigma^{-p}S_{A})$.

2) Since $A$ has homology concentrated in non negative degrees, then $M\in\cd^+A$. Therefore, $X=\sigma_{\leq p}M$ has bounded homology. Note that the hypothesis implies that each $H^nM\ko n\in\Z$, is finitely generated as a module over $H^0A$. This implies that each $H^nX\ko n\in\Z$, is finitely generated as a module over $H^0A$. Now we can use part~1) of the proposition.
\end{proof}

\begin{corollary}\label{lifting graded simples}
Let $k$ be a commutative associative ring with unit. Let $A$ be a dg $k$-algebra such that:
\begin{itemize}
\item[a)] $H^pA$ vanishes for $p<0$,
\item[b)] $\Mod H^0(A)$ is a semisimple abelian category.
\end{itemize}
Then for each graded simple module $S$ over the graded ring $H^*A$, there is a dg $A$-module $\tilde{S}$, unique up to isomorphism in the derived category $\cd A$, such that the graded $H^*A$-module $H^*(\tilde{S})$ is isomorphic to $S$.
\end{corollary}
\begin{proof}
{\em First step: The graded simple modules over $H^*A$ are precisely the simple modules over $H^0A$, regarded as graded $H^*A$-modules (concentrated in degree $0$) by restricting scalars along $H^*A\ra H^0A$.} Clearly, simple $H^0A$-modules become simple graded $H^*A$-modules. Conversely, if $S$ is a graded simple $H^*A$-module, then it has to be concentrated in degree $0$. This implies that it is killed by $\bigoplus_{p>0}H^pA$. In other words, it is a (necessarily simple) $H^0A$-module.

{\em Second step: There exists a decomposition into indecomposables $A=\bigoplus_{i=1}^{r}A_{i}$ of $A$ in $\cd A$ such that $\H0A=\bigoplus_{i=1}^{r}\H0(A_{i})$ is a decomposition into simples of $\H0A$ in $\Mod\H0(A)$.} This is Lemma~\ref{indecomposable decomposition of A}.

{\em Third step: the graded $H^*A$-modules $H^*(\sigma_{\leq 0} A_{i})\ko 1\leq i\leq r$, are graded simple $H^*A$-modules, and every graded simple $H^*A$-module is of this form.} Thanks to the first step, it suffices to prove that $H^p(\sigma_{\leq 0}A_{i})=0$ for $p\neq 0$, and that with $H^0(\sigma_{\leq 0}A_{i})\ko 1\leq i\leq r$, we get all the simple $H^0A$-modules. This follows from the particular properties of the weight structure we are considering.

{\em Fourth step: if $\tilde{S}\in\cd A$ is a module such that $H^*(\tilde{S})$ is a graded $H^*A$-module isomorphic to $H^*(\sigma_{\leq 0}A_{i})$ for some $1\leq i\leq r$, then $\tilde{S}$ is isomorphic to $\sigma_{\leq 0}(A_{i})$ in $\cd A$.} Indeed, the proof of part~1 of Proposition~\ref{the subcategory generated by the simple} can be used to show that the map $\tilde{f}:\sigma_{\leq 0}(A_{i})\ra \tilde{S}$ there is an isomorphism.
\end{proof}

\begin{remark}
The result above remains valid for small dg categories $\ca$ such
that $H^p\ca=0$ for $p<0$ and $\Mod H^0(\ca)$ is semi-simple and
each simple is compact.
\end{remark}

\section{The Koszul dual}

Throughout this section $A$ will be a dg algebra with homology
concentrated in non negative degrees and such that $H^0A$ is a
semi-simple ring. Recall from Notation~\ref{notation canonical
weight structure} that $S_{A}=\sigma_{\leq 0}(A)$.

\begin{notation}
We write $B=\REnd(S_{A})$. It should be thought thought of as the
`Koszul dual' of $A$.
\end{notation}

\begin{lemma}
$B$ has homology concentrated in non positive degrees.
\end{lemma}
\begin{proof}
We have to prove that
\[H^p\RHom(S_{A},S_{A})=\Hom_{\cd A}(S_{A},\Sigma^pS_{A})=0
\]
for $p>0$. After applying $\Hom_{\cd A}(?,\Sigma^pS_{A})$ to the triangle
\[\sigma_{>0}(A)\ra A\ra S_{A}\ra\Sigma\sigma_{>0}(A)
\]
we get the exact sequence
\[\Hom(\sigma_{>0}(A),\Sigma^{p-1}S_{A})\ra\Hom(S_{A},\Sigma^pS_{A})\ra H^p(S_{A}).
\]
Of course, $H^p(S_{A})=0$ for $p>0$. On the other hand, by definition of weight structure we have
\[\Hom(\sigma_{>0}(A),\Sigma^{p-1}S_{A})=0
\]
for $p>0$.
\end{proof}

\begin{lemma}\label{maps as lim-colim}
\begin{itemize}
\item[1)] For each $X\in\cd A$ we have $X\cong\op{Mlim}_{p\geq 0}\sigma_{\leq p}X$.
\item[2)] For every pair of objects $X$ and $Y$ of $\cd A$ we have
\[\Hom(X,Y)=\op{lim}_{q}\op{colim}_{p}\Hom(\sigma_{\leq p}X,\sigma_{\leq q}Y).
\]
\end{itemize}
\end{lemma}
\begin{proof}
1) Given $X\in\cd A$ we can form triangles
\[\sigma_{>0}(X)\ra X\ra\sigma_{\leq 0}(X)\ra\Sigma\sigma_{>0}(X),
\]
\[\sigma_{>1}(\sigma_{>0}X)\ra\sigma_{>0}X\ra\sigma_{\leq 1}(\sigma_{>0}X)\ra\Sigma\sigma_{>1}(\sigma_{>0}X),
\]
\[\dots
\]
Thanks to statement (2) of Theorem~\ref{canonical weight structure},
we can take all these triangles so that the maps induce isomorphisms
at the level of convenient homologies. Using the octahedron axiom of
triangulated categories we prove that in the triangle
\[\sigma_{>1}\sigma_{>0}X\ra X\ra C\ra \Sigma\sigma_{>1}\sigma_{>0}X,
\]
over the composition
\[\sigma_{>1}(\sigma_{>0}X)\ra \sigma_{>0}(X)\ra X,
\]
the object $C$ belongs to $(\cd A)^{w\leq 1}$. Thus
\[\sigma_{>1}\sigma_{>0}X\ra  X\ra  C\ra  \Sigma\sigma_{>1}\sigma_{>0}X
\]
is the truncation triangle corresponding to the weight structure $((\cd A)^{w\leq 1},(\cd A)^{w\geq 1})$, and we can write $C=\sigma_{\leq 1}(X)$ and $\sigma_{>1}(\sigma_{>0})X=\sigma_{>1}(X)$. Moreover, we still have an isomorphism
\[H^pX\arr{\sim}H^p(\sigma_{\leq 1}X)
\]
for $p\leq 1$. Indeed, for $p\leq 0$ we have the following diagram with exact rows
\[\xymatrix{ 0\ar[r] & H^pX\ar[r]\ar[d] & H^p(\sigma_{\leq 0}X)\ar[r]\ar@{=}[d] & H^{p+1}(\sigma_{>0}X)\ar[d]^{\wr} \\
0\ar[r] & H^p(\sigma_{\leq 1}X)\ar[r] & H^p(\sigma_{\leq 0}X)\ar[r] & H^{p+1}(\sigma_{\leq 1}\sigma_{>0}X),
}
\]
and for $p=1$ we have the following diagram with exact rows
\[\xymatrix{H^1(\sigma_{>1}X)\ar[r]\ar@{=}[d] & H^1(\sigma_{>0}X)\ar[r]\ar[d]^{\wr} & H^1(\sigma_{\leq 1}\sigma_{>0}X)\ar[r]\ar[d] & H^2(\sigma_{>1}X)\ar[r]\ar@{=}[d] & H^2(\sigma_{>0}X)\ar[d]^{\wr} \\
H^1(\sigma_{>1}X)\ar[r] & H^1X\ar[r] & H^1(\sigma_{\leq 1}X)\ar[r] & H^2(\sigma_{>1}X)\ar[r] & H^2X,
}
\]
which implies that $H^1(\sigma_{\leq 1}\sigma_{>0}X)\ra H^1(\sigma_{\leq 1}X)$ is an isomorphism, and so from the square
\[\xymatrix{H^1(\sigma_{>0}X)\ar[r]^{\sim}\ar[d]^{\wr} &H^1(\sigma_{\leq 1}\sigma_{>0}X)\ar[d]^{\wr} \\
H^1X\ar[r] & H^1(\sigma_{\leq 1}X)
}
\]
we deduce that $H^1X\ra H^1(\sigma_{\leq 1}X)$ is an isomorphism.

Repeating this construction we get a commutative diagram
\[\xymatrix{\dots\ar[r] & \sigma_{\leq 2}X\ar[r] & \sigma_{\leq 1}X\ar[r] & \sigma_{\leq 0}X \\
\dots\ar[r] & X\ar@{=}[r]\ar[u]^{g_{2}} & X\ar@{=}[r]\ar[u]^{g_{1}} & X\ar[u]^{g_{0}} \\
\dots\ar[r] & \sigma_{>2}X\ar[r]\ar[u] & \sigma_{>1}X\ar[r]\ar[u] & \sigma_{>0}X\ar[u]
}
\]
where the morphisms $H^n(g_{p}):H^nX\ra H^n(\sigma_{\leq p}X)$ are isomorphisms for $n\leq p$. Consider now the induced map
\[X\ra\op{Mlim}_{p\geq 0}\sigma_{\leq p}X.
\]
For each $n\in\Z$ we get a map
\[H^nX\ra H^n(\op{Mlim}_{p\geq 0}\sigma_{\leq p}X)=\op{lim}_{p\geq 0}H^n(\sigma_{\leq p}X)
\]
induced by
\[\xymatrix{\dots\ar[r] & H^n(\sigma_{\leq 2}X)\ar[r] & H^n(\sigma_{\leq 1}X)\ar[r] & H^n(\sigma_{\leq 0}X) \\
\dots\ar[r] & H^nX\ar@{=}[r]\ar[u]^{H^ng_{2}} & H^nX\ar@{=}[r]\ar[u]^{H^ng_{1}} & H^nX\ar[u]^{H^ng_{0}} \\
\dots\ar[r] & H^n(\sigma_{>2}X)\ar[r]\ar[u] & H^n(\sigma_{>1}X)\ar[r]\ar[u] & H^n(\sigma_{>0}X)\ar[u]
}
\]
For each $n\in\Z$, almost every map $H^n(g_{p})$ is an isomorphism, and so the map $H^nX\ra H^n(\op{Mlim}_{p\geq 0}\sigma_{\leq p}X)$ is an isomorphism.

2) Given $X\ko Y\in\cd A$, we have $Y=\op{Mlim}_{q\geq 0}\sigma_{\leq q}Y$, and so
\[\Hom(X,Y)=\Hom(X,\op{Mlim}_{q}\sigma_{\leq q}Y)=\op{lim}_{q}\Hom(X,\sigma_{\leq q}Y).
\]
After applying $\Hom(?,\sigma_{\leq q}Y)$ to the commutative diagram (see the proof of part~1))
\[\xymatrix{
\dots\ar[r] & \Sigma\sigma_{>2}X\ar[r] & \Sigma\sigma_{>1}X\ar[r] & \Sigma\sigma_{>0}X \\
\dots\ar[r] & \sigma_{\leq 2}X\ar[r]\ar[u] & \sigma_{\leq 1}X\ar[r]\ar[u] & \sigma_{\leq 0}X\ar[u] \\
\dots\ar[r] & X\ar@{=}[r]\ar[u] & X\ar@{=}[r]\ar[u] & X\ar[u] \\
\dots\ar[r] & \sigma_{>2}X\ar[r]\ar[u] & \sigma_{>1}X\ar[r]\ar[u] & \sigma_{>0}X\ar[u]
}
\]
we get the commutative diagram
\[\xymatrix{
\Hom(\Sigma\sigma_{>0}X,\sigma_{\leq q}Y)\ar[r]\ar[d] & \Hom(\Sigma\sigma_{>1}X,\sigma_{\leq q}Y)\ar[r]\ar[d] & \Hom(\Sigma\sigma_{>2}X,\sigma_{\leq q}Y)\ar[r]\ar[d] &  \dots \\
\Hom(\sigma_{\leq 0}X,\sigma_{\leq q}Y)\ar[d]\ar[r] & \Hom(\sigma_{\leq 1}X,\sigma_{\leq q}Y)\ar[d]\ar[r] & \Hom(\sigma_{\leq 2}X,\sigma_{\leq q}Y)\ar[d]\ar[r] &  \dots \\
\Hom(X,\sigma_{\leq q}Y)\ar[d]\ar@{=}[r] & \Hom(X,\sigma_{\leq q}Y)\ar@{=}[r]\ar[d] &\Hom(X,\sigma_{\leq q}Y)\ar@{=}[r]\ar[d] &  \dots \\
\Hom(\sigma_{>0}X,\sigma_{\leq q}Y)\ar[r] & \Hom(\sigma_{>1}X,\sigma_{\leq q}Y)\ar[r] & \Hom(\sigma_{>2}X,\sigma_{\leq q}Y)\ar[r] &  \dots
}
\]
For $p\gg 0$ we have $\Hom(\sigma_{>p}X,\sigma_{\leq
q}Y)=0=\Hom(\Sigma\sigma_{>p}X,\sigma_{\leq q}Y)$, and so the map
$\Hom(\sigma_{\leq p}X,\sigma_{\leq q}Y)\ra\Hom(X,\sigma_{\leq q}Y)$
is an isomorphism. Hence,
\[\Hom(X,Y)=\op{lim}_{q\geq 0}\Hom(X,\sigma_{\leq q}Y)=\op{lim}_{q\geq 0}\op{colim}_{p\geq 0}\Hom(\sigma_{\leq p}X,\sigma_{\leq q}Y).
\]
\end{proof}

\begin{proposition}\label{fully faithful}
Assume that each $H^pA\ko p\in\Z$, is a finitely generated $H^0A$-module. Then the functor
\[\RHom(?,S_{A}): (\per A)^{\text{op}}\ra \cd(B^{\text{op}}),
\]
which has its image in $\cd^-(B^{\text{op}})$, is fully faithful.
\end{proposition}
\begin{proof}
For the first claim it suffices to notice that
\[\Hom_{\cd A}(\Sigma^{-p}X,\sigma_{\leq 0}A)=0
\]
for $X\in\per A$ and $p\gg 0$, since every object in $\per A$ has left bounded homology.

We prove the second claim in several steps.

{\em First step: The functor $\RHom(?,S_{A}):\op{thick}(S_{A})^{\text{op}}\ra\cd(B^{\text{op}})$ is fully faithful.} Indeed, we can do finite d\'evissage using the fact that the map
\[\RHom(?,S_{A}): \Hom_{\cd A}(S_{A},S_{A})\ra \Hom_{\cd(B^\text{op})}(B,B)
\]
is an isomorphism.

{\em Second step: preservation of truncation of perfect objects.} Here we will use both the weight structure on $\cd A$ (see Corollary~\ref{canonical weight structure for positive dg algebra}) and the canonical weight structure on $\cd(B^\text{op})$ (see Appendix 0).  The truncation triangle for $A$ corresponding to the weight structure of Corollary~\ref{canonical weight structure for positive dg algebra} is
\[\sigma_{>0}(A)\ra A\ra S_{A}\ra\Sigma\sigma_{>0}(A).
\]
After applying $\RHom(?,S_{A})$ and rotating we gt the triangle
\[B\ra\RHom(A,S_{A})\ra\RHom(\sigma_{>0}A,S_{A})\ra \RHom(S_{A},\Sigma S_{A}),
\]
where $B\in\cd^-(B^\text{op})^{w\geq 0}$ and $\RHom(\sigma_{>0}A,S_{A})\in\cd^-(B^\text{op})^{w<0}$. If $X$ is an arbitrary perfect module, then
one can prove that $\RHom(\sigma_{\leq p}X,S_{A})$ belongs to $\cd^-(B^\text{op})^{w\geq -p}$ by using part~2) of Proposition~\ref{the subcategory generated by the simple} together with Remark~\ref{shifts of the regular module in the right wing}, and one can prove that $\RHom(\sigma_{>p}X,S_{A})\in\cd^-(B^\text{op})^{w<p}$ by using the orthogonality property of weight structures.

{\em Third step: the claim.} Put $F=\RHom(?,S_{A})$. Let $X$ and $Y$ be two objects of $\per A$. Thanks to Lemma~\ref{maps as lim-colim}, Theorem~\ref{canonical weight structure for negative dg algebras}, step 2 and step 1 of this proof, and Proposition~\ref{the subcategory generated by the simple}, we have the following commutative diagram
\[\xymatrix{\Hom(X,Y)\ar@{=}[dd]\ar[rrr] &&& \Hom(FY, FX)\ar@{=}[d] \\
 &&& \op{lim}_{q}\op{colim}_{p}\Hom(\sigma_{\geq -q}FY, \sigma_{\geq -p}FX)\ar@{=}[d] \\
 \op{lim}_{q}\op{colim}_{p}\Hom(\sigma_{\leq p}X,\sigma_{\leq q}Y)\ar[rrr]^{\sim} &&& \op{lim}_{q}\op{colim}_{p}\Hom(F\sigma_{\leq q}Y, F\sigma_{\leq p}X)
}
\]
\end{proof}

\section{Reminder on $t$-structures}

 A {\em $t$-structure}
\cite{BBD} on a triangulated category $\cd$ is a pair $t=(\cd^{\leq
0},\cd^{\geq 0})$ of strictly full triangulated subcategories of
$\cd$ such that:
\begin{itemize}
\item[1)] $\cd^{\leq 0}$ is closed under $\Sigma$ and $\cd^{\geq 0}$ is closed under $\Sigma^{-1}$,
\item[2)] $\Hom_{\cd}(M,\Sigma^{-1}N)=0$ for each $M\in\cd^{\leq 0}$ and $N\in\cd^{\geq 0}$,
\item[3)] for each $M\in\cd$ there exists a triangle in $\cd$
\[M^{\leq 0}\ra M\ra M^{\geq 1}\ra\Sigma M^{\leq 0},
\]
with $M^{\leq 0}\in\cd^{\leq 0}$ and $\Sigma (M^{\geq 1})\in\cd^{\geq 0}$.
\end{itemize}

It is easy to prove that each one of the two subcategories completely determines the other one in the following sense: an object $N\in\cd$ belongs to $\cd^{\geq 0}$ (resp. $\cd^{\leq 0}$) if and only if we have
\[\Hom_{\cd}(M,\Sigma^{-1}N)=0
\]
for each $M\in\cd^{\leq 0}$ (resp. for each $N\in\cd^{\geq 0}$).

It is also easy to prove that the triangle above is unique up to a unique isomorphism extending the identity morphism $\id_{M}$. Hence, for each $M\in\cd$ we can make choices of the objects $M^{\leq 0}$ and $M^{\geq 1}$ so that the map $M\mapsto M^{\leq 0}$ underlies a functor $(?)^{\leq 0}:\cd\ra\cd^{\leq 0}$ right adjoint to the inclusion, and the map $M\mapsto \Sigma((\Sigma^{-1}M)^{\geq 1})$ underlies a functor $(?)^{\geq 0}:\cd\ra\cd^{\geq 0}$ left adjoint to the inclusion.

The {\em heart} of $t$ is the full subcategory $\ch(t)$ of $\cd$ formed by those objects which are in $\cd^{\leq 0}$ and also in $\cd^{\geq 0}$. It is an abelian category, and the functor
\[\H{0}:\cd\ra\ch(t)\ko M\mapsto (M^{\leq 0})^{\geq 0},
\]
which is said to be the {\em $0$th homology functor of $t$}, is homological, \ie takes triangles to long exact sequences.

A $t$-structure $t=(\cd^{\leq 0},\cd^{\geq 0})$ is {\em non degenerate} if we have
\[\bigcap_{n\in\Z}\Sigma^n\cd^{\leq 0}=\{0\}=\bigcap_{n\in\Z}\Sigma^n\cd^{\geq 0}.
\]
This property implies that an object $M$ of $\cd$:
\begin{enumerate}
\item[-] vanishes if and only if $\H{0}(\Sigma^nM)=0$ for each $n\in\Z$,
\item[-] belongs to $\cd^{\leq 0}$ if and only if $\H0(\Sigma^nM)=0$ for $n>0$,
\item[-] belongs to $\cd^{\geq 0}$ if and only if $\H0(\Sigma^nM)=0$ for $n<0$.
\end{enumerate}
The $t$-structure $t$ is {\em bounded} if we have:
\[\bigcup_{n\in\Z}\Sigma^n\cd^{\leq 0}=\cd=\bigcup_{n\in\Z}\Sigma^n\cd^{\geq 0}.
\]
Note that any bounded $t$-structure $t$ is non degenerate. Indeed,
if $t$ is bounded, any object $M$ is a finite extension of shifts of
objects of the form $H^0(\Sigma^nM)\ko n\in\Z$. But if
$M\in\bigcap_{n\in\Z}\Sigma^n\cd^{\leq 0}$ or
$M\in\bigcap_{n\in\Z}\Sigma^n\cd^{\geq 0}$, then we have
$H^0(\Sigma^nM)=0$ for each $n\in\Z$.

A {\em left aisle} (resp. {\em right aisle}) in a triangulated category $\cd$ is a full subcategory $\cu$ containing a zero object $0$ of $\cd$, closed under $\Sigma$ (resp. $\Sigma^{-1}$), closed under extensions, and such that the inclusion functor $\cu\ra\cd$ admits a right (resp. left) adjoint. We have already mentioned that if $t=(\cd^{\leq 0},\cd^{\geq 0})$ is a $t$-structure on $\cd$, then $\cd^{\leq 0}$ is a left aisle in $\cd$ and $\cd^{\geq 0}$ is a right aisle in $\cd$. Moreover, it is proved in \cite[\S 1]{KellerVossieck1988b} that the map $(\cd^{\leq 0},\cd^{\geq 0})\mapsto\cd^{\leq 0}$ underlies a bijection between the set of $t$-structures on $\cd$ and the set of left aisles in $\cd$, and similarly for right aisles. We will refer to $\cd^{\leq 0}$ (resp. $\cd^{\geq 0}$) as {\em the left (resp. right) aisle of $t$}.

\begin{example}\label{canonical t-structure for negative dg algebras}
It is shown in Appendix~2 that if $A$ is a dg algebra, there exists
a $t$-structure $t_{A}$ on its unbounded derived category $\cd A$
such that $\cd^{\geq 0}$ is formed by those modules whose ordinary
homology is concentrated in non negative degrees, and $\cd^{\leq 0}$
is formed by those modules $M$ which fit into a triangles
\[\coprod_{i\geq 0}L_{i}\ra\coprod_{i\geq 0}L_{i}\ra M\ra\Sigma\coprod_{i\geq 0}L_{i}
\, ,
\]
where $L_{i}$ is an $i$-fold extension of small coproducts of
non negative shifts of $A$. Therefore, if $A$ has homology
concentrated in non positive degrees, it is not difficult to prove
that $\cd^{\leq 0}$ is formed by those modules whose ordinary
homology is concentrated in non positive degrees. In this case, if
we assume moreover, as we may, that the components of $A$ vanish in
strictly positive degrees,  the functors $(?)^{\leq 0}$ and
$(?)^{\geq 0}$ are given by the usual intelligent truncations, and
the associated $0$th homology functor gives the ordinary homology in
degree $0$. Therefore, we say that the $t$-structure $t_{A}$ is the
{\em canonical} one. It is a non degenerate $t$-structure, whose
heart is equivalent to the category of unital right modules over the
ring $\H0(A)$:
\[\H{0}:\ch(t_{A})\arr{\sim}\Mod\H0(A)
\]
(see for example \cite[Lemma 5.2.b)]{KellerYang}).

Assume now that $A$ is a dg algebra over a field $k$, and let us consider the finite-dimensional derived category $\cd_{fd}A$ (see \S~\ref{notation}). The
canonical $t$-structure on $\cd A$ restricts to a bounded $t$-structure $t_{A}^{fd}$ on $\cd_{fd}A$, whose heart is
equivalent to the category of finite-dimensional unital right
modules over the $k$-algebra $\H0(A)$:
\[\H{0}:\ch(t^{fd}_{A})\arr{\sim}\mod\H0(A).
\]
In particular, $\ch(t_{A}^{fd})$ is a length category. If, moreover,
$\H0(A)$ is finite-dimensional, then $\ch(t^{fd}_{A})$ has a finite
number of isoclasses of simple objects.
\end{example}

\section{Application to the construction of $t$-structures}\label{s:construction-of-t-structures}

\begin{theorem}\label{theorem:t-structure on perA}
Let $k$ be a commutative associative ring with unit, and let $A$ be a dg $k$-algebra such that:
\begin{itemize}
\item[a)] $H^pA=0$ for $p< 0$,
\item[b)] $H^0A$ is a semi-simple $k$-algebra, and
\item[c)] each $H^pA\ko p\in\Z$, is a finitely generated $H^0A$-module.
\end{itemize}
Then the perfect derived category $\per A$ admits a bounded $t$-structure whose left (resp. right) aisle is the smallest full subcategory containing $A$ and closed under extensions, positive (resp. negative) shifts and direct summands. Its heart is a length category whose simple objects are the indecomposable direct summands of $A$ in $\per A$.
\end{theorem}

\begin{remark} Suppose $k$ is a field and $H^0A$ is isomorphic to a product of copies of $k$. Then the theorem follows from Proposition~3.4 of Rickard-Rouquier's \cite{RickardRouquier10} applied to the triangulated category $\ct=\per A$ and to the set $\mathcal{S}$ formed by a system of representatives of the indecomposable direct factors of $A$ in $\per A$.
\end{remark}

\begin{proof}
Consider the functor $\RHom(?,S_{A}):(\cd A)^{\text{op}}\ra\cd(B^\text{op})$. Thanks to Proposition~\ref{fully faithful}, we know its restriction to $(\per A)^{\text{op}}$ is fully faithful. Notice that the obvious morphism of dg algebras $B\ra H^0B$ (using the intelligent truncation) and the isomorphism of ordinary algebras $H^0B\ra H^0A$ allow us to regard $H^0A$ as a dg $B$-module. Moreover, we have isomorphisms
\[\RHom(A,S_{A})\arr{\sim} S_{A}{\stackrel{\sim}{\leftarrow}}H^0A
\]
compatible with the structure of left dg $B$-modules of $\RHom(A,S_{A})$ and $H^0A$. Thus
\[\RHom(?,S_{A}):(\per A)^{\text{op}}\arr{\sim}\op{thick}_{\cd(B^{\text{op}})}(H^0A)
\]
is an equivalence. The picture of the situation is the following:
\[\xymatrix{(\cd A)^{\text{op}}\ar[d]_{\RHom(?,S_{A})} && (\per A)^{\text{op}}\ar[d]^{\wr}\ar@{^(->}[ll] \\
\cd(B^{\text{op}}) && \op{thick}(H^0A)\ar@{^(->}[ll]
}
\]
Let us consider a full subcategory $\ca$ of the heart $\ch$ of the canonical $t$-structure on $\cd(B^{\text{op}})$ formed by those objects with a finite composition series in which the composition factors are direct summands of $H^0A$. It is not difficult to prove that $\op{thick}(H^0A)$ is precisely the full subcategory $\ct$ of $\cd(B^\text{op})$ formed by those modules $M$ such that $H^pM=0$ for almos every $p\in\Z$ and $H^pM\in\ca$ for each $p\in\Z$. With this description it is easy to check that the canonical $t$-structure restricts to a $t$-structure on $\ct$ whose heart is $\ca$. The simple objects of this heart are given by the simple $H^0A$-modules, \ie the indecomposable direct summands of $H^0A$, which corresponds bijectively to the indecomposable direct summands of $A$.
\end{proof}

A triangulated category can be recovered from the heart of a bounded $t$-structure by closing under extensions and shifts. Taking this into account, we have:

\begin{corollary}
Let $A$ be as in Theorem \ref{theorem:t-structure on perA}. Then $\per A$ is the smallest full triangulated subcategory of $\cd A$ closed under extensions, shifts and containing the indecomposable direct summands of $A$ .
\end{corollary}

\begin{remark}
Notice that the simple objects of the heart are also in bijection with the simple modules over $\H0(A)$.
\end{remark}

\begin{corollary}\label{formal}
Let $A$ be an algebra as in Theorem \ref{theorem:t-structure on
perA}. If we assume moreover that $A$ is formal, then
$\per(H^*A)$ admits a canonical $t$-structure whose left
(resp. right) aisle is the smallest full subcategory containing
$H^*A$ and closed under extensions, positive (resp. negative) shifts and direct summands. Its heart is a length category
whose simples are the indecomposable direct summands of
$H^*A$ in $\per(H^*A)$.
\end{corollary}

\begin{remark}
Theorem \ref{theorem:t-structure on perA} should be compared with a
result by O.~Schn\"urer \cite{Schnuerer08} which states the
existence of a canonical $t$-structure on the perfect derived
category of a dg algebra $B$ positively graded, with $B^0$
semi-simple and whose differential vanishes on $B^0$. The main
motivation for Schn\"urer's theorem was to prove that certain
categories of sheaves, endowed with a perverse $t$-structure, are
$t$-equivalent to the perfect derived category of a certain dg
algebra $B$ endowed with a canonical $t$-structure (see
\cite{SchnuererThesis}). In practice, $B$ is the homology algebra
$H^*A$ of a formal dg algebra $A$ satisfying conditions of
Theorem \ref{theorem:t-structure on perA}, and so the existence of a
canonical $t$-structure on $\per B$ follows from Theorem
\ref{theorem:t-structure on perA} and Corollary \ref{formal}.

\end{remark}

\begin{example}\label{example of positive dg algebra}
Let $A$ be a dg algebra over a field $k$ such that in each degree its homology is of finite dimension and vanishes for large degrees. Let $S_{1}\ko\dots\ko S_{r}$, be a family of perfect $A$-modules such that:
\begin{itemize}
\item[a)] $\Hom_{\cd A}(S_{i},S_{j})=\begin{cases}0 & \text{ if }i\neq j,\\ k\cdot \id_{S_{i}}&\text{ if }i=j.\end{cases}$
\item[b)] $\Hom_{\cd A}(S_{i},\Sigma^pS_{j})=0$ for each $p<0$.
\end{itemize}
Then the derived endomorphism dg algebra
$B=\op{REnd}_{A}(\bigoplus_{i=1}^{r}S_{i})$ satisfies the conditions
of Theorem~\ref{theorem:t-structure on perA}. Indeed, the homology
groups of $B$ vanish in degrees $<0$ by condition~b) and they are
finite-dimensional and vanish in degrees $\gg 0$ because the $S_i$
are perfect.
\end{example}

\begin{nonexample}
Here we show that condition~b) of our theorem is not redundant.
Indeed, let $A$ be a finite-dimensional algebra of infinite global
dimension over a field $k$. We will show that $\per A$ does not
admit a canonical $t$-structure. Indeed, assume $\per A$ admits a
$t$-structure $t$ such that $\per(A)^{t\leq 0}$ is the smallest full
subcategory of $\per A$ containing $A$ and closed under extensions,
shifts and direct summands. Then, by d\'evissage, we deduce that
$\per(A)^{t\geq 0}$ is the full subcategory of $\per A$ formed by
those objects with ordinary homology concentrated in non negative
degrees. On the other hand, it is clear that the objects of
$\per(A)^{t\leq 0}$ have ordinary homology concentrated in
non positive degrees. Thus, if $P$ belongs to $\per A$, then in the
triangle
\[
P^{t\leq 0} \ra P \ra P^{t\geq 1} \ra\Sigma (P^{t\leq 0}) \ko
\]
the object $P^{t\leq 0}$ only has homology in non positive degrees
and the object $P^{t\geq 1}$ only has homology in strictly positive
degrees. Therefore, this is the triangle for the natural
$t$-structure and so the truncation functors of the given
$t$-structure $t$ on $\per A$ coincide with those of the natural
$t$-structure. It follows that $\per A$ is stable under the natural
truncation functor $P \mapsto \tau_{\geq 0} P$. This is a
contradiction since we may take $P=(P_1 \to P_0)$ to be the
beginning of a projective resolution of an $A$-module of infinite
projective dimension. Thus, $\per A$ does not admit a canonical
$t$-structure.
\end{nonexample}


\section{Application to hearts and simple-minded objects}
\label{s:hearts}

Let $k$ be an algebraically closed field, and let $A$ be a dg $k$-algebra such that:
\begin{itemize}
\item[1)] in each degree its homology is of finite dimension,
\item[2)] its homology vanishes for large degrees,
\item[3)] $A$ is {\em homologically smooth}, \ie $A$ is a compact object of the unbounded derived category of dg $A$-$A$-bimodules.
\end{itemize}

\begin{remark}
Note that these conditions are invariant under derived Morita equivalence. The reader can find the proof of the invariance of condition~3) in \cite[Lemma 2.6]{ToenVaquie}.
\end{remark}

\begin{example}
Let $A$ be an ordinary finite-dimensional algebra over a perfect
field $k$. Then $A$ is homologically smooth if and only if it has
finite global dimension. That the finiteness of the global dimension
is necessary already appeared in Cartan-Eilenberg's book
\cite[Proposition IX.7.6]{CartanEilenberg}. That it is a sufficient
condition can be proved by using, for example, the ideas of the
proof of \cite[Lemma 1.5]{HappelSeminaireAlgebre}.
\end{example}

\begin{example}
We can also take $A$ to be the non complete Ginzburg dg algebra associated to a Jacobi-finite quiver with potential \cite{Ginzburg} \cite{KellerYang}. The fact that in this case $A$ satisfies condition~3) has been proved in \cite{KellerDeformedCalabiYauCompletions}. That condition~1) also holds has been proved in \cite{Amiot}.
\end{example}

Following Rickard (unpublished) and Koenig-Liu \cite{KoenigLiu}, we
define a {\em family of simple-minded objects} to be a finite family
$S_{1}\ko\dots\ko S_{r}$ of objects of $\cd_{fd}A$ such that:
\begin{itemize}
\item[a)] $\Hom_{\cd A}(S_{i},S_{j})=\begin{cases}0 & \text{ if }i\neq j,\\ k\cdot \id_{S_{i}}&\text{ if }i=j.\end{cases}$
\item[b)] $\Hom_{\cd A}(S_{i},\Sigma^tS_{j})=0$ for each $t<0$.
\item[c)] $\cd_{fd}A$ is the smallest full triangulated subcategory of $\cd A$ containing the objects $S_{1}\ko\dots\ko S_{r}$.
\end{itemize}

\begin{example}\label{getting families of simple-minded}
Let $t$ be a bounded $t$-structure on $\cd_{fd}A$ whose heart
$\ch(t)$ is a length category with a finite number of isoclasses of
simple objects. Then we can take $S_{1}\ko\dots\ko S_{r}$ to be a
family of representatives of those isoclasses.
\end{example}

Two families $S_{1}\ko\dots\ko S_{r}$ and $S'_{1}\ko\dots\ko S'_{r'}$ of simple-minded objects of $\cd_{fd}A$ are {\em equivalent} if they have the same closure under extensions.

\begin{corollary}\label{bounded t-structures in bijection with simple-minded objects}
Taking representatives of the isoclasses of the simple objects of the heart yields a bijection between:
\begin{itemize}
\item[1)] Bounded $t$-structrures on $\cd_{fd}A$ whose heart is a length category with a finite number of isoclasses of simple objects.
\item[2)] Equivalence classes of families of simple-minded objects of $\cd_{fd}(A)$.
\end{itemize}
\end{corollary}
\begin{proof}
{\em First step: from $t$-structures to simple-minded objects.} We have already observed in Example~\ref{getting families of simple-minded} that, from such a $t$-structure on $\cd_{fd}A$, one gets a family of simple-minded objects of $\cd_{fd}A$ by considering the simples of the corresponding heart.

{\em Second step: from simple-minded objects to $t$-structures.} Conversely, let $S_{1}\ko\dots\ko S_{r}$ be a family of simple-minded objects of $\cd_{fd}A$. Put $S=\bigoplus_{i=1}^{r}S_{i}$ and $B=\op{REnd}_{A}(S)$. The adjoint pair
\[\xymatrix{\cd A\ar@<1ex>[d]^{\RHom_{A}(S,?)} \\
\cd B\ar@<1ex>[u]^{?\otimes^\L_{B}S}
}
\]
induces mutually quasi-inverse triangle functors
\[\xymatrix{\cd_{fd}A\ar@<1ex>[d]^{\RHom_{A}(S,?)} \\
\per B.\ar@<1ex>[u]^{?\otimes^\L_{B}S} }
\]
Under these equivalences, the objects $S_{i}$ correspond to the
indecomposable direct summands of $B$ in $\per B$. As noticed in
Example~\ref{example of positive dg algebra}, $B$ satisfies the
hypothesis of Theorem~\ref{theorem:t-structure on perA}. Therefore,
there exists a bounded $t$-structure on $\per B$ whose heart is a
length category such that the indecomposable direct summands of $B$
in $\per B$ are the representatives of the isoclasses of the simple
objects. This $t$-structure is mapped by $?\otimes^\L_{B}S$ to a
bounded $t$-structure on $\cd_{fd}A$ whose heart is a length
category such that the simple-minded objects we started with are the
representatives of the isoclasses of the simple objects.

{\em Third step: the bijection.} By using that a bounded $t$-structure is completely determined by its heart (see for example \cite[Lemma 2.3]{Bridgeland2005}) it is easy to check that steps 1 and 2 define a bijection.
\end{proof}

\begin{corollary}
$S_{1}\ko\dots\ko S_{r}$ and $S'_{1}\ko\dots\ko S'_{r'}$ are two equivalent families of simple-minded objects of $\cd_{fd}A$ if and only if $r=r'$ and, up to reordering, $S_{i}\cong S'_{i}$.
\end{corollary}
\begin{proof}
After Corollary \ref{bounded t-structures in bijection with
simple-minded objects}, two equivalent families of simple-minded
objects are families of representatives of the isoclasses of the
simple modules of the same length category.
\end{proof}

\section{Appendix 0: A weight structure for negative dg algebras}\label{a weight structure for negative dg algebras}

Let $B$ be a dg algebra with homology concentrated in non positive degrees. Consider the following full subcategories of $\cd B$:
\begin{itemize}
\item $\cd^{w\leq 0}$, formed by those modules with homology concentrated in non positive degrees,
\item $\cd^{w\geq 0}$, formed by those modules $X$ satisfying $\Hom(X,Y)=0$ for each $Y\in\cd^{w<0}=\Sigma\cd^{w\leq 0}$.
\end{itemize}

\begin{remark}\label{shifts of the regular module in the right wing}
Note that $\Sigma^pB\in\cd^{w\geq 0}$ for each $p\leq 0$.
\end{remark}

The following result is an unbounded analogue of a result
by Bondarko, cf. \S 6 of \cite{Bondarko0704.4003v8}.

\begin{theorem}\label{canonical weight structure for negative dg algebras}
\begin{itemize}
\item[1)] The pair $(\cd^{w\leq 0},\cd^{w\geq 0})$ is a weight structure on $\cd B$.
\item[2)] $\cd^{w\leq 0}$ is the smallest full subcategory $\op{Susp}(B)$ of $\cd B$ containing $B$ and closed under positive shifts, extensions and arbitrary coproducts.
\item[3)] For any object $X$ of $\cd B$ we have $X\cong\op{Mcolim}_{p\geq 0}\sigma_{\geq -p}X$.
\item[4)] For any pair $X$ and $Y$ of objects of $\cd B$ we have
\[\Hom(X,Y)=\op{lim}_{q\geq 0}\op{colim}_{p\geq 0}\Hom(\sigma_{\geq -q}X,\sigma_{\geq -p}Y).
\]
\end{itemize}
\end{theorem}
\begin{proof}
2) It is clear that $B\in\cd^{w\leq 0}$, and that $\cd^{w\leq 0}$ is closed under extensions, positive shifts and arbitrary coproducts. Therefore $\op{Susp}(B)$ is contained in $\cd^{w\leq 0}$. Now, for an object $M$ of $\cd^{w\leq 0}$  we can form a sequence of triangles
\[B_{0}\arr{u} M\arr{v} M_{0}\ra \Sigma B_{0},
\]
\[B_{1}\ra M_{0}\ra M_{1}\ra\Sigma B_{1},
\]
\[\dots
\]
by taking $B_{p}=\coprod_{q\geq p}\coprod_{\Hom(\Sigma^qB,M_{p})}\Sigma^qB$ and defining $B_{p}\ra M_{p}$ as the obvious map. This yields a diagram
\[\xymatrix{M_{0}\ar[r] & M_{1}\ar[r] & M_{2}\ar[r] & \dots \\
M\ar[u]^{v}\ar[r]^{\id} & M\ar[u]\ar[r]^{\id} & M\ar[u]\ar[r] & \dots \\
L_{0}\ar[u]^{u}\ar[r] & L_{1}\ar[u]\ar[r] & L_{2}\ar[u]\ar[r] & \dots
}
\]
where each $L_{p}$ is a $p$-fold extension of coproducts of non negative shifts of $B$. Thanks to Verdier's $3\times 3$ lemma (see \cite[Proposition 1.1.11]{BBD}) we know there exists a triangle
\[L\ra M\ra \Mcolim M_{p}\ra \Sigma L,
\]
where $L$ fits in a triangle of the form
\[\coprod_{p\geq 0}L_{p}\ra L\ra \coprod_{p\geq 0}\Sigma L_{p}\ra \Sigma\coprod_{p\geq 0}L_{p}.
\]
Thus, it is clear that $L\in\op{Susp}(B)$. On the other hand, for each $n\geq 0$ we have
\[\Hom(\Sigma^{n}B,\Mcolim M_{p})=\op{colim}\Hom(\Sigma^nB, \Mcolim M_{p})=0
\]
because the morphisms
\[\Hom(\Sigma^nB,M_{p})\ra\Hom(\Sigma^nB,M_{p+1})
\]
vanish. Thus $\Mcolim M_{p}$ has homology concentrated in degrees $\geq 1$. But, in fact, for each $n\geq 1$ we have an exact sequence
\[H^nM\ra H^n(\Mcolim M)\ra H^{n+1}L,
\]
where $H^nM=0$ by hypothesis and $H^{n+1}L=0$ because $B$ has homology concentrated in non positive degrees. This proves that $\Mcolim M_{p}=0$, and so $M\cong L\in\op{Susp}(B)$.

1) It is clear that $\cd^{w\leq 0}$ and $\cd^{w\geq 0}$ are closed under finite coproducts and direct summands. It is also clear that $\cd^{w\leq 0}$ is closed under positive shifts and $\cd^{w\geq 0}$ is closed under negative shifts. The ortohogonality axiom hols by definition of $\cd^{w\geq 0}$. It remains to prove the existence of a truncation triangle. Let $M$ be an object of $\cd B$. Thanks to \cite[\S 3.1]{Keller1994a} we can assume that $M$ has a filtration
\[0=M_{-1}\subset M_{0}\subset M_{1}\subset\dots\subset M_{n-1}\subset M_{n}\dots\subset M
\]
in the category $\cc B$ of dg $B$-modules such that
\begin{itemize}
\item[F1)] $M=\colim_{n\geq 0}M_{n}$,
\item[F2)] each $M_{n-1}\ra M_{n}$ in an inflation in $\cc B$, \ie it is a degreewise split-injection,
\item[F3)] $M_{n}/M_{n-1}$ is a small coproduct of (positive or negative) shifts of $B$.
\end{itemize}

Using the fact that $B$ has homology concentrated in non positive degrees, we can form a commutative square
\[\xymatrix{L'_{1}\ar[r]\ar[d] & \Sigma L_{0}\ar[d] \\
M_{1}/M_{0}\ar[r] & \Sigma M_{0}
}
\]
where the vertical morphisms are degree-wise split injections and $L'_{1}$ (resp. $L_{0}$) is the direct summand of $M_{1}/M_{0}$ (resp. $M_{0}$) formed by the non positive shifts of $B$. Taking the co-cone $L_{1}$ of $L'_{1}\ra \Sigma L_{0}$ we get a morphism of degree-wise split short exact sequences of dg $B$-modules
\[\xymatrix{L_{0}\ar[r]\ar[d] & L_{1}\ar[r]\ar[d] & L'_{1}\ar[d] \\
M_{0}\ar[r] & M_{1}\ar[r] & M_{1}/M_{0},
}
\]
where the vertical arrows are degree-wise split injections. We write $L'_{1}=L_{1}/L_{0}$. In this way, we can form morphisms of degree-wise split short exact sequences of dg $B$-modules
\[\xymatrix{L_{n-1}\ar[r]\ar[d] & L_{n}\ar[r]\ar[d] & L_{n}/L_{n-1}\ar[d] \\
M_{n-1}\ar[r] & M_{n}\ar[r] & M_{n}/M_{n-1}
}
\]
for each $n\geq 0$, where the vertical arrows are degree-wise split injections and $L_{n}/L_{n-1}$ is the direct summand of $M_{n}/M_{n-1}$ formed by the non positive shifts of $B$. This yields a sequence of degree-wise split short exact sequences of dg $B$-modules
\[\xymatrix{0=L_{-1}\ar[r]\ar[d] & L_{0}\ar[r]\ar[d] & L_{1}\ar[r]\ar[d] & \dots & L_{n-1}\ar[r]\ar[d] & L_{n}\ar[d] & \dots \\
0=M_{-1}\ar[r]\ar[d] & M_{0}\ar[r]\ar[d] & M_{1}\ar[r]\ar[d] & \dots & M_{n-1}\ar[r]\ar[d] & M_{n}\ar[d] & \dots \\
0=N_{-1}\ar[r] & N_{0}\ar[r] & N_{1}\ar[r] & \dots & N_{n-1}\ar[r] & N_{n} & \dots,
}
\]
where for each $n\geq 0$ there is a morphisms of degree-wise split short exact sequences of dg $B$-modules
\[\xymatrix{M_{n-1}\ar[r]\ar[d] & M_{n}\ar[r]\ar[d] & M_{n}/M_{n-1}\ar[d] \\
N_{n-1}\ar[r] & N_{n}\ar[r] & N_{n}/N_{n-1}
}
\]
where the vertical arrows are degree-wise split surjections and $N_{n}/N_{n-1}$ is the direct summand of $M_{n}/M_{n-1}$ formed by the positive shifts of $B$. Write $L=\op{colim}_{n\geq 0}L_{n}$ and $N=\op{colim}_{n\geq 0}N_{n}$. The short exact sequence of dg $B$-modules
\[0\ra L\ra M\ra N\ra 0
\]
induces a triangle
\[L\ra M\ra N\ra \Sigma L
\]
in $\cd B$. Note that $L=\Mcolim_{n\geq 0}L_{n}$ and $N=\Mcolim_{n}N_{n}$. Since $\cd^{w<0}$ is closed under small coproducts, positive shifts and extensions, then $N\in\cd^{w<0}$. On the other hand, if $Y\in\cd^{w<0}$ then
\[\Hom(L,Y)=\op{lim}_{n\geq 0}\Hom(L_{n},Y)=0,
\]
which proves that $L\in\cd^{w\geq 0}$.

3) We can construct a commutative diagram as follows:
\[\xymatrix{\sigma_{\geq 0}X\ar[r]\ar[d]^{f_{0}} & \sigma_{\geq -1}X\ar[r]\ar[d]^{f_{-1}} & \sigma_{\geq -2}X\ar[r]\ar[d]^{f_{-2}} & \dots \\
X\ar@{=}[r]\ar[d] & X\ar@{=}[r]\ar[d] & X\ar@{=}[r]\ar[d] & \dots \\
\sigma_{<0}X\ar[r] & \sigma_{<-1}X\ar[r] & \sigma_{>-2}X\ar[r] & \dots
}
\]
which induces a morphism
\[f:\op{Mcolim}_{p\geq 0}\sigma_{\geq -p}X\ra X.
\]
For each $n\in\Z$ this yields a morphism
\[H^n(f):\op{colim}_{p\geq 0}H^n(\sigma_{\geq -p}X)\ra H^nX
\]
induced by the commutative diagram
\[\xymatrix{H^n(\sigma_{\geq 0}X)\ar[r]\ar[d]^{H^n(f_{0})} & H^n(\sigma_{\geq -1}X)\ar[r]\ar[d]^{H^n(f_{-1})} & H^n(\sigma_{\geq -2}X)\ar[r]\ar[d]^{H^n(f_{-2})} & \dots \\
H^nX\ar@{=}[r]\ar[d] & H^nX\ar@{=}[r]\ar[d] & H^nX\ar@{=}[r]\ar[d] & \dots \\
H^n(\sigma_{<0}X)\ar[r] & H^n(\sigma_{<-1}X)\ar[r] & H^n(\sigma_{>-2}X)\ar[r] & \dots
}
\]
We deduce that $H^n(f)$ is an isomorphism from the fact that almost every map $H^n(f_{-p})\ko p\geq 0$, is an isomorphism.

4) Note that we have
\[\Hom(X,Y)=\Hom(\op{Mcolim}_{q\geq 0}\sigma_{\geq -q}X,Y)=\op{lim}_{q\geq 0}\Hom(\sigma_{\geq -q}X,Y).
\]
Now for a fix $q\geq 0$, we apply $\Hom(\sigma_{\geq -q}X,?)$ to the diagram
\[\xymatrix{\sigma_{\geq 0}Y\ar[r]\ar[d] & \sigma_{\geq -1}Y\ar[r]\ar[d] & \sigma_{\geq -2}Y\ar[r]\ar[d] & \dots \\
Y\ar@{=}[r]\ar[d] & Y\ar@{=}[r]\ar[d] & Y\ar@{=}[r]\ar[d] & \dots \\
\sigma_{<0}Y\ar[r] & \sigma_{<-1}Y\ar[r] & \sigma_{>-2}Y\ar[r] & \dots
}
\]
to get the diagram
\[\xymatrix{\Hom(\sigma_{\geq -q}X,\sigma_{\geq 0}Y)\ar[r]\ar[d] & \Hom(\sigma_{\geq -q}X,\sigma_{\geq -1}Y)\ar[r]\ar[d] & \Hom(\sigma_{\geq -q}X,\sigma_{\geq -2}Y)\ar[r]\ar[d] & \dots \\
\Hom(\sigma_{\geq -q}X,Y)\ar@{=}[r]\ar[d] & \Hom(\sigma_{\geq -q}X,Y)\ar@{=}[r]\ar[d] & \Hom(\sigma_{\geq -q}X,Y)\ar@{=}[r]\ar[d] & \dots \\
\Hom(\sigma_{\geq -q}X,\sigma_{<0}Y)\ar[r] & \Hom(\sigma_{\geq -q}X,\sigma_{<-1}Y)\ar[r] & \Hom(\sigma_{\geq -q}X,\sigma_{>-2}Y)\ar[r] & \dots
}
\]
in which $\Hom(\sigma_{\geq -q}X,\sigma_{<-p}Y)=0$ for $p\gg 0$. Thus the induced morphism
\[\op{colim}_{p\geq 0}\Hom(\sigma_{\geq -q}X,\sigma_{\geq -p}Y)\ra\Hom(\sigma_{\geq -q}X,Y)
\]
is an isomorphism.
\end{proof}

\section{Appendix 1: Milnor colimits versus homotopy colimits}\label{Appendix 1}

Let $\mathbb{D}$ be a triangulated derivator defined on the
$2$-category of small categories (see \cite{CisinskiNeeman2005} and
the notation therein). Let us denote by $e$ the $1$-point category.
For any small category $I$, we will write $p:I\ra e$ to refer to the
unique possible functor. We have an adjoint pair of triangle
functors
\[\xymatrix{\mathbb{D}(e)\ar@<1ex>[d]^{p^*} \\
\mathbb{D}(I)\ar@<1ex>[u]^{p_{!}}
}
\]
and, by definition, if $F\in\mathbb{D}(I)$ we say that $p_{!}F$ is the {\em homotopy colimit} of $F$. Sometimes this will be denoted by $\op{hocolim}F$ or $\Gamma_{!}(F,I)$.

In this Appendix, we will show that if a triangulated category $\cd$
is at the base of a triangulated derivator, then Milnor colimits of
sequences of morphisms of $\cd$ are isomorphic to homotopy colimits.

The key tool will be the diagram functor (see \cite[\S 1.10]{CisinskiNeeman2005}):
\[d_{I}:\mathbb{D}(I)\ra\ul{\Hom}(I^{\text{op}},\mathbb{D}(e))
\]
(sometimes we shall omit the subscript $I$). If $F$ is an object of
$\mathbb{D}(I)$, we say that $d_{I}F$ is the {\em diagram} or {\em
presheaf} associated to $F$. Given a presheaf
$F\in\ul{\Hom}(I^{\text{op}},\mathbb{D}(e))$, we say that an object
$G\in\mathbb{D}(I)$ is {\em lifts} $F$ if $d_{I}(G)$ is isomorphic
to $F$ in $\ul{\Hom}(I^{\text{op}},\mathbb{D}(e))$.

For each $i\in I$, we denote by $?\otimes i:\mathbb{D}(e)\ra\ul{\Hom}(I^{\text{op}},\mathbb{D}(e))$ the left adjoint of the functor $(?)_{i}$ {\em evaluation at $i$}:
\[\xymatrix{\ul{\Hom}(I^{\text{op}},\mathbb{D}(e))\ar@<1ex>[d]^{(?)_{i}} \\
\mathbb{D}(e)\ar@<1ex>[u]^{?\otimes i}
}
\]
For $j\in I$ and $X$ in $\mathbb{D}(e)$, we have the canonical
isomorphism
\[
(X\ten i)_j = \coprod_{\Hom(j,i)} X.
\]
\begin{lemma} For each $i$ in $I$, the triangle
\[
\xymatrix{
\mathbb{D}(I)\ar[r]^-{d_{I}} & \ul{\Hom}(I^{\text{op}},\mathbb{D}(e)) \\
\mathbb{D}(e)\ar[u]^{i_{!}}\ar[ur]_{?\otimes i} &
}
\]
commutes up to a canonical isomorphism.
\end{lemma}

\begin{proof} Recall that by axiom Der4d, for each functor $u: J \to I$
and each object $j$ of $I$, we have a canonical isomorphism
\[
j^* u_! = p_! \, l^* \ko
\]
where the functors are those of the square
\[
 \xymatrix{ j\backslash J \ar[r]^{l} \ar[d]_p & J \ar[d]^u \\
e \ar[r]_j & I}
\]
and $j\backslash J$ is the comma-category of pairs
$(j', u(j') \to j)$. Let us specialize $J$ to $e$ and $u$ to
the inclusion determined by the object $i$ of $I$. Then
we get a canonical isomorphism
\[
j^* i_! = p_! \, p^* \ko
\]
where now $i\backslash J = i\backslash e$ is the discrete
category $\Hom(j,i)$ and $p$ the unique functor $\Hom(j,i) \to e$.
By axiom Der1, the composition $p_! \, p^*$ is the coproduct
composed with the diagonal functor. So for each object $X$ of
$\mathbb{D}(e)$, we get a canonical isomorphism
\[
(i_! X)_j = \coprod_{\Hom(j,i)} X = (X\ten i)_j.
\]
One checks that these isomorphisms yield a canonical
isomorphism as claimed.
\end{proof}

\begin{remark} \label{left adjoint to the evaluation}
For $I^{\text{op}}=\N$ and $n\in\N$, the object
$X\otimes n$ is the presheaf
\[0\ra\dots\ra 0\ra X\arr{\id}X\arr{\id}X\ra\dots,
\]
where the first $X$ appears in position $n$ and by the lemma,
the triangle
\[
\xymatrix{
\mathbb{D}(\N^{\text{op}})\ar[r]^{d_{\N^{\text{op}}}\hspace{0.5cm}} & \ul{\Hom}(\N,\mathbb{D}(e)) \\
\mathbb{D}(e)\ar[u]^{n_{!}}\ar[ur]_{?\otimes n} &
}
\]
commutes up to isomorphism.
\end{remark}

\begin{proposition}\label{LiftingMilnorHocolim}
\begin{itemize}
\item[1)] Given an object $X$ of $\ul{\Hom}(\N,\mathbb{D}(e))$, there exists an object of $\mathbb{D}(\N^{\text{op}})$ which lifts $X$.
\item[2)] Given a morphism $f:X\ra X'$ in $\ul{\Hom}(\N,\mathbb{D}(e))$ there exists a morphism $\tilde{f}:\tilde{X}\ra\tilde{X}'$ in $\mathbb{D}(\N^{\text{op}})$ such that $d_{\N^{\text{op}}}(\tilde{f})$ is isomorphic to $f$.
\item[3)] The homotopy colimit of an object $X$ of $\mathbb{D}(\N^{\text{op}})$ is isomorphic to the Milnor colimit of its associated diagram $d_{\N^{\text{op}}}X$.
\end{itemize}
\end{proposition}
\begin{proof}
1) {\em Step 1: an exact category with global dimension $1$.} Every additive category can be endowed with an exact structure by taking as conflations the split exact pairs (see \cite{KellerDCU} and the terminology therein). Let us consider $\mathbb{D}(e)$ as an exact category in this way, and let us regard $\ul{\Hom}(\N,\mathbb{D}(e))$ as an exact category with the pointwise split exact structure. Let us calculate a projective resolution of an arbitrary object of this category. Given an object $X$ of $\ul{\Hom}(\N,\mathbb{D}(e))$, \ie a sequence of morphisms in $\mathbb{D}(e)$
\[X_{0}\arr{x_{0}}X_{1}\arr{x_{1}}X_{2}\arr{x_{2}}\dots,
\]
we can start the projective resolution by considering the deflation
\[P_{0}=\coprod_{n\in N}X_{n}\otimes n\ra X
\]
defined by using the counit of the adjunctions $(?\otimes n,(?)_{n})$. It turns out that $\ul{\Hom}(\N,\mathbb{D}(e))$ has global dimension $1$. Indeed, in the kernel $P_{1}\arr{u}P_{0}$ of the former deflation we can take
\[P_{1}=\coprod_{n\in\N}X_{n}\otimes(n+1),
\]
which is a projective object. An explicit diagram might help
\[\xymatrix{0\ar[r]\ar[d] & X_{0}\ar[r]\ar[d]^{\scriptsize{\left[\begin{array}{c}\id \\ -x_{0}\end{array}\right]}} & X_{0}\oplus X_{1}\ar[r]\ar[d]^{\scriptsize{\left[\begin{array}{cc}\id & 0 \\ -x_{0} & \id \\ 0 & -x_{1}\end{array}\right]}} & \dots \\
X_{0}\ar[r]\ar[d] & X_{0}\oplus X_{1}\ar[r]\ar[d]^{\scriptsize{\left[\begin{array}{cc}x_{0}&\id\end{array}\right]}} & X_{0}\oplus X_{1}\oplus X_{2}\ar[r]\ar[d]^{\scriptsize{\left[\begin{array}{ccc}x_{1}x_{0}&x_{1}&\id\end{array}\right]}} & \dots \\
X_{0}\ar[r]^{x_{0}} & X_{1}\ar[r]^{x_{1}} & X_{2}\ar[r] & \dots
}
\]

{\em Step 2: lifting a projective resolution along the diagram functor.} Put
\[\tilde{P}_{1}=\coprod_{n\in\N}(n+1)_{!}(X_{n})
\]
and
\[\tilde{P}_{0}=\coprod_{n\in\N}n_{!}(X_{n}).
\]
For each $n\in\N$, let $a_{n}\in\Hom_{\mathbb{D}(\N^{\text{op}})}((n+1)_{!}X,n_{!}X)$ be the image of the identity $\id_{n_{!}(X_{n})}$ by the composition of the morphisms
\begin{align}
\Hom_{\mathbb{D}(\N^{\text{op}})}(n_{!}(X_{n}),n_{!}(X_{n}))\arr{\sim}\Hom_{\mathbb{D}(e)}(X_{n},n^*n_{!}(X_{n})) \nonumber \\
\ra\Hom_{\mathbb{D}(e)}(X_{n},(n+1)^*n_{!}(X_{n})) \nonumber \\
\arr{\sim}\Hom_{\mathbb{D}(\N^{\text{op}})}((n+1)_{!}(X_{n}),n_{!}(X_{n})) \nonumber
\end{align}
induced by the adjoint pairs $(n_{!},n^*)$ and $((n+1)_{!},(n+1)^*)$ and the $2$-arrow
\[\xymatrix{\mathbb{D}(e) && \hspace{0.5cm}\mathbb{D}(\N^{\text{op}})\lltwocell<5>_{(n+1)^*}^{n^*}{\hspace{1.7cm}(\alpha^{n+1}_{n})^*}.
}
\]
coming from the only possible $2$-arrow
\[\xymatrix{e\rrtwocell<5>^{n+1}_{n}{\hspace{0.65cm}\alpha^{n+1}_{n}} && \N^{\text{op}}.
}
\]
Consider now the morphism
\[\tilde{P}_{1}\arr{\tilde{u}} \tilde{P}_{0}
\]
in $\mathbb{D}(\N^{\text{op}})$ determined by
\[\xymatrix{
\tilde{P}_{1}\ar[rrrr]^{\tilde{u}} &&&& \tilde{P}_{0} \\
(n+1)_{!}(X_{n})\ar[u]\ar[rrrr]^{\scriptsize{\left[\begin{array}{cc}a_{n} & -(n+1)_{!}(x_{n})\end{array}\right]^{t}}\hspace{1cm}} &&&& n_{!}(X_{n})\oplus(n+1)_{!}(X_{n+1})\ar[u]
}
\]
Remark \ref{left adjoint to the evaluation} tells us that the diagram functor $d_{\N^{\text{op}}}$ sends $\tilde{u}$ to $u:P_{1}\ra P_{0}$.

{\em Step 3: a triangle over the lifted morphism.} Now consider a triangle
\[\tilde{P}_{1}\arr{\tilde{u}}\tilde{P}_{0}\ra \tilde{X}\ra \Sigma\tilde{P}_{1}
\]
in $\mathbb{D}(\N^{\text{op}})$. For each $m\in\N$, after applying the triangle functor $m^*:\mathbb{D}(\N^{\text{op}})\ra\mathbb{D}(e)$ we get a triangle
\[\xymatrix{
\bigoplus_{n=0}^{m-1}X_{n}\ar[r]^{u_{m}}  & \bigoplus_{n=0}^mX_{n}\ar[r] & d_{\N^{\text{op}}}(\tilde{X})_{m}\ar[r] & \Sigma \bigoplus_{n=0}^{m-1}X_{n}
}
\]
in $\mathbb{D}(e)$. Since $u_{m}$ is a section, $d_{\N^{\text{op}}}(\tilde{X})_{m}$ is the cokernel of $u_{m}$ and so $d_{\N^{\text{op}}}(\tilde{X})_{m}\cong X_{m}$.

2) Given a morphism $f:X\ra X'$ in $\ul{\Hom}(\N,\mathbb{D}(e))$, we can consider as before the projective resolutions
\[P_{1}\arr{u}P_{0}\ra X
\]
and
\[P'_{1}\arr{u'}P'_{0}\ra X'.
\]
By using $f:X\ra X'$ we can define a morphism $g:P_{0}\ra P'_{0}$ making commutative the square
\[\xymatrix{P_{0}\ar[r]\ar[d]^g & X\ar[d]^f \\
P'_{0}\ar[r] & X
}
\]
and the universal property of the cokernel guarantees the existence of a morphism of conflations
\[\xymatrix{P_{1}\ar[r]^{u}\ar[d]^{h} & P_{0}\ar[r]\ar[d]^g & X\ar[d]^f \\
P'_{1}\ar[r]^{u'} & P'_{0}\ar[r] & X
}
\]
Thanks to Remark \ref{left adjoint to the evaluation}, we can prove that there exists a commutative square
\[\xymatrix{\tilde{P}_{1}\ar[r]^{\tilde{u}}\ar[d]^{\tilde{h}} & \tilde{P}_{0}\ar[d]^{\tilde{g}} \\
\tilde{P}'_{1}\ar[r]^{\tilde{u}'} & \tilde{P}'_{0}
}
\]
in $\mathbb{D}(\N^{\text{op}})$ which is mapped to
\[\xymatrix{P_{1}\ar[r]^{u}\ar[d]^{h} & P_{0}\ar[d]^g\\
P'_{1}\ar[r]^{u'} & P'_{0}
}
\]
by $d_{\N^{\text{op}}}$. The commutative square in $\mathbb{D}(\N^{\text{op}})$ can be completed to a morphism of triangles
\[\xymatrix{
\tilde{P}_{1}\ar[r]^{\tilde{u}}\ar[d]^{\tilde{h}} & \tilde{P}_{0}\ar[d]^{\tilde{g}}\ar[r] & \tilde{X}\ar[r]\ar[d]^{\tilde{f}} & \Sigma\tilde{P}_{1}\ar[d] \\
\tilde{P}'_{1}\ar[r]^{\tilde{u}'} & \tilde{P}'_{0}\ar[r] & \tilde{X}'\ar[r] & \Sigma\tilde{P}'_{1}
}
\]
For each $m\in\N$, we apply the triangle functor $m^*$ and obtain a morphism of triangles
\[\xymatrix{\bigoplus_{n=0}^{m-1}X_{n}\ar[r]^{u_{m}}\ar[d] & \bigoplus_{n=0}^mX_{n}\ar[r]\ar[d] & (d_{\N^{\text{op}}}\tilde{X})_{m}\ar[d]\ar[r] & \Sigma \bigoplus_{n=0}^{m-1}X_{n}\ar[d] \\
\bigoplus_{n=0}^{m-1}X'_{n}\ar[r]^{u'_{m}} & \bigoplus_{n=0}^mX'_{n}\ar[r] & (d_{\N^{\text{op}}}\tilde{X}')_{m}\ar[r] & \Sigma \bigoplus_{n=0}^{m-1}X'_{n}
}
\]
Since both $u_{m}$ and $u'_{m}$ are sections, $(d_{\N^{\text{op}}}\tilde{X})_{m}$ is the cokernel of $u_{m}$ and $(d_{\N^{\text{op}}}\tilde{X}')_{m}$ is the cokernel of $u'_{m}$. Thus, $d_{\N^{\text{op}}}(\tilde{f})_{m}$ is isomorphic to $f_{m}$.

3) Given an object $X\in\mathbb{D}(\N^{\text{op}})$ we consider a triangle
\[Y\ra \coprod_{n\in\N}n_{!}n^*X\arr{\varepsilon}X\ra\Sigma Y
\]
where $\varepsilon$ is defined by using the counit of the adjunctions $(n_{!},n^*)$.
For each $n\in\N$, let $a_{n}\in\Hom_{\mathbb{D}(\N^{\text{op}})}((n+1)_{!}n^*X,n_{!}n^*X)$ be the image of the identity $\id_{n_{!}n^*X}$ by the composition of the morphisms
\begin{align}
\Hom_{\mathbb{D}(\N^{\text{op}})}(n_{!}n^*X,n_{!}n^*X)\arr{\sim}\Hom_{\mathbb{D}(e)}(n^*X,n^*n_{!}n^*X) \nonumber \\
\ra\Hom_{\mathbb{D}(e)}(n^*X,(n+1)^*n_{!}n^*X) \nonumber \\
\arr{\sim}\Hom_{\mathbb{D}(\N^{\text{op}})}((n+1)_{!}n^*X,n_{!}n^*X) \nonumber
\end{align}
induced by the adjoint pairs $(n_{!},n^*)$ and $((n+1)_{!},(n+1)^*)$ and the $2$-arrow
\[\xymatrix{\mathbb{D}(e) && \hspace{0.5cm}\mathbb{D}(\N^{\text{op}})\lltwocell<5>_{(n+1)^*}^{n^*}{\hspace{1.7cm}(\alpha^{n+1}_{n})^*}.
}
\]
coming from the only possible $2$-arrow
\[\xymatrix{e\rrtwocell<5>^{n+1}_{n}{\hspace{0.65cm}\alpha^{n+1}_{n}} && \N^{\text{op}}.
}
\]
Consider the morphism
\[\coprod_{n\in\N}(n+1)_{!}n^*X\arr{u}\coprod_{n\in\N}n_{!}n^*X
\]
described by
\[\xymatrix{\coprod_{n\in\N}(n+1)_{!}n^*X\ar[rr]^{u} && \coprod_{n\in\N}n_{!}n^*X \\
(n+1)_{!}n^*X\ar[u]\ar[rr]^{\scriptsize{\left[\begin{array}{cc} a_{n} & -x_{n}\end{array}\right]^{t}}\hspace{1.3cm}} && n_{!}n^*X\oplus (n+1)_{!}(n+1)^*X\ar[u]
}
\]
Since the composition $\varepsilon u$ vanishes, there exists a morphism $\varphi$ making commutative the diagram
\[\xymatrix{& \coprod_{n\in\N}(n+1)_{!}n^*X\ar[d]^{u}\ar[dl]_{\varphi}\ar[dr]^{0} && \\
Y\ar[r] & \coprod_{n\in\N}n_{!}n^*X\ar[r]^{\hspace{0.5cm}\varepsilon} & X\ar[r] & \Sigma Y
}
\]
For each $m\in\N$, after applying the triangle functor $m^*$ we get a triangle
\[\xymatrix{
m^*Y\ar[r] & \coprod_{n\in\N}m^*n_{!}n^*X\ar[r]^{\hspace{0.8cm}m^*\varepsilon} & m^*X\ar[r] & \Sigma m^*Y
}
\]
By using Remark \ref{left adjoint to the evaluation} we know that
\[\coprod_{n\in\N}m^*n_{!}n^*X=\bigoplus_{n=0}^{m}n^*X,
\]
and it is easy to check that the $n$th composite of the morphism $m^*\varepsilon$ is the morphism
\[(\alpha^m_{n})^*:n^*X\ra m^*X
\]
given by the unique $2$-arrow
\[\xymatrix{e\rrtwocell<5>^{m}_{n}{\hspace{0.3cm}\alpha^m_{n}} && \N^{\text{op}}.
}
\]
Thus, $m^*\varepsilon$ is a section, with retraction given by
\[\scriptsize{\left[\begin{array}{cccc}0 & \dots & 0 & \id\end{array}\right]^t}:m^*X\ra\bigoplus_{n=0}^{m}n^*X.
\]
>From this, we deduce that the morphism
\[m^*Y\ra \bigoplus_{n=0}^{m}n^*X
\]
is the kernel of $m^*\varepsilon$. On the other hand, it is easy to check that the kernel of $m^*\varepsilon$ is $m^*u$. Therefore, $m^*\varphi$ is an isomorphism for each $m\in\N$, and the conservative axiom of derivators (see \cite[Definition 1.11]{CisinskiNeeman2005}) says that $\varphi$ is an isomorphism. Finally, if we apply the triangle functor $\op{hocolim}$ to the triangle
\[\xymatrix{\coprod_{n\in\N}(n+1)_{!}n^*X\ar[r]^{\hspace{0.3cm}u} & \coprod_{n\in\N}n_{!}n^*X\ar[r]^{\hspace{0.75cm}\varepsilon} & X\ar[r] & \Sigma \coprod_{n\in\N}(n+1)_{!}n^*X
}
\]
we get the triangle
\[\xymatrix{\coprod_{n\in\N}n^*X\ar[r]^{\id-\sigma} & \coprod_{n\in\N}n^*X\ar[r] & \op{hocolim}X\ar[r] & \Sigma \coprod_{n\in\N}n^*X.
}
\]
The $n$th composite of $\sigma$ is the composition
\[\xymatrix{
n^*X\ar[rr]^{(\alpha^{n+1}_{n})^*\hspace{1.8cm}} && (n+1)^*X\ra\coprod_{n\in\N}n^*X,
}
\]
where $\alpha^{n+1}_{n}$ is the only possible $2$-arrow
\[\xymatrix{e\rrtwocell<5>^{n+1}_{n}{\hspace{0.5cm}\alpha^{n+1}_{n}} && \N^{\text{op}}.
}
\]
Therefore,
\[\op{hocolim}X\cong\op{Mcolim}d_{\N^{\text{op}}}X.
\]
\end{proof}

If $X$ is an object of $\ul{\Hom}(\N,\mathbb{D}(e))$ given by
\[X_{0}\arr{x_{0}} X_{1}\arr{x_{1}} X_{2}\ra \dots,
\]
we denote by $\Sigma X$ the object $\ul{\Hom}(\N,\mathbb{D}(e))$ given by
\[\Sigma X_{0}\arr{\Sigma x_{0}} \Sigma X_{1}\arr{\Sigma x_{1}} \Sigma X_{2}\ra \dots
\]

If $\cd$ is a triangulated category and $f:X\ra Y$ is a morphism in the category $\ul{\Hom}(\N,\cd)$:
\[\xymatrix{X_{0}\ar[r]\ar[d]^{f_{0}} & X_{1}\ar[r]\ar[d]^{f_{1}} & X_{2}\ar[r]\ar[d]^{f_{2}} & \dots \\
Y_{0}\ar[r] & Y_{1}\ar[r] & Y_{2}\ar[r] & \dots,
}
\]
we write $\op{Mcolim}f$ to refere to a morphism which fits in a morphism of triangles
\[\xymatrix{\coprod_{n\in\N}X_{n}\ar[r]^{\id-\sigma}\ar[d]^{\coprod_{n\in\N}f_{n}} & \coprod_{n\in\N}X_{n}\ar[r]\ar[d]^{\coprod_{n\in\N}f_{n}} & \op{Mcolim}X_{n}\ar[r]\ar[d]^{\op{Mcolim}f} & \Sigma\coprod_{n\in\N}X_{n}\ar[d] \\
\coprod_{n\in\N}Y_{n}\ar[r]^{\id-\sigma} & \coprod_{n\in\N}Y_{n}\ar[r] & \op{Mcolim}Y_{n}\ar[r] & \Sigma\coprod_{n\in\N}Y_{n}
}
\]

\begin{corollary}\label{Milnor colimits are quasi exact}
Let
\[X\arr{f}Y\ra Z\ra \Sigma X
\]
be a diagram in $\ul{\Hom}(\N,\mathbb{D}(e))$ such that for each $n\in\N$ the corresponding diagram
\[X_{n}\arr{f_{n}}Y_{n}\ra Z_{n}\ra \Sigma X_{n}
\]
is a triangle in $\mathbb{D}(e)$. There exists a triangle
\[\xymatrix{
\op{Mcolim}X\ar[rr]^{\op{Mcolim}f} && \op{Mcolim}Y\ar[r] & \op{Mcolim}Z'\ar[r] & \Sigma\op{Mcolim}X
}
\]
in $\mathbb{D}(e)$, where $Z'$ is an object of $\ul{\Hom}(\N,\mathbb{D}(e))$ such that $Z'_{n}\cong Z_{n}$ for each $n\in\N$.
\end{corollary}
\begin{proof}
Part~2) of Proposition \ref{LiftingMilnorHocolim} tells us that there exists a morphism $\tilde{f}:\tilde{X}\ra\tilde{Y}$ in $\mathbb{D}(\N^{\text{op}})$ such that $d_{\N^{\text{op}}}(\tilde{f})=f$. Let us complete this morphism to a triangle
\[\tilde{X}\arr{\tilde{f}}\tilde{Y}\ra\tilde{Z}\ra\Sigma\tilde{X}
\]
in $\mathbb{D}(\N^{\text{op}})$. For a natural number $n\in\N$ the triangle functor $n^*$ sends this triangle to a triangle
\[X_{n}\arr{f_{n}}Y_{n}\ra n^*\tilde{Z}\ra\Sigma X_{n},
\]
which proves that $n^*\tilde{Z}\cong Z_{n}$. On the other hand, by using part~1) of Proposition \ref{LiftingMilnorHocolim} we get that the triangle functor $\op{hocolim}$ sends the triangle in $\mathbb{D}(\N^{\text{op}})$ to a triangle
\[\xymatrix{\op{Mcolim}X\ar[rr]^{\op{Mcolim}f} && \op{Mcolim}Y\ar[r] & \op{hocolim}\tilde{Z}\ar[r] & \Sigma\op{Mcolim}X.
}
\]
Finally, part~3) of Proposition \ref{LiftingMilnorHocolim} tells us that
\[\op{hocolim}\tilde{Z}\cong\op{Mcolim}d_{\N^{\text{op}}}(\tilde{Z}).
\]
\end{proof}

\section{Appendix 2: From compact objects to $t$-structures}

It is well known that from a set $\cs$ of compact objects of a triangulated category $\cd$ with small coproducts one can
produce in a natural way an interesting $t$-structure $t_{\cs}$. For example, in \cite[Theorem III.2.3]{BeligiannisReiten},
it is proved that if $\cy_{\cs}$ is the full subcategory of $\cd$ formed by those objects $Y$ such that
$\Hom_{\cd}(\Sigma^nS,Y)=0$ for each $n\geq 0$ and each $S\in\cs$, then $\cy_{\cs}$ is the right aisle of a $t$-structure. In fact, this can be deduced from \cite[Theorem A.1]{AlonsoJeremiasSouto2003}. For the convenience of the reader we will include here the statement and the proof of that theorem:

\begin{theorem}\label{AlonsoJeremiasSouto}
Let $\cd$ be a triangulated category with small coproducts, and let $\cs$ be a set of compact objects of $\cd$. Then:
\begin{itemize}
\item[1)] the smallest full subcategory $\op{Susp}_{\cd}(\cs)$ of $\cd$ containing $\cs$ and closed under extensions, positive shifts and small coproducts is a left aisle,
\item[2)] every object $X$ of $\op{Susp}_{\cd}(\cs)$ fits in a triangle
\[
\coprod_{i\geq 0}X_{i}\ra X\ra \coprod_{i\geq 0}\Sigma X_{i}\ra \coprod_{i\geq 0}\Sigma X_{i}
\]
where $X_{i}$ is an $i$-fold extension of small coproducts of non negative shifts of objects of $\cs$.
\end{itemize}
\end{theorem}
\begin{proof}
Let $M$ be an object of $\cd$, and let us consider an approximation
\[P_{0}\ra M
\]
of $M$ with respect to the full subcategory of $\cd$ formed by the small coproducts of non negative shifts of objects of $\cs$. Let us consider a triangle
\[P_{0}\arr{f_{0}} M\arr{g_{0}} Y_{0}\ra\Sigma P_{0}
\]
and a new approximation
\[P_{1}\ra Y_{0}
\]
with respect to the same subcategory. By iterating this procedure we get a diagram of the form
\[\xymatrix{M\ar[r]^{g_{0}} & Y_{0}\ar[r]^{y_{0}}\ar@{~>}[dl] & Y_{1}\ar@{~>}[dl]  \\
P_{0}\ar[u]^{f_{0}} & P_{1}\ar[u] & \dots
}
\]
This diagram yields a diagram
\[\xymatrix{\Sigma X_{0}\ar[r]^{\Sigma x_{0}} & \Sigma X_{1}\ar[r]^{\Sigma x_{1}} & \Sigma X_{2}\ar[r] & \dots \\
Y_{0}\ar[r]^{y_{0}}\ar[u] & Y_{1}\ar[r]^{y_{1}}\ar[u] & Y_{2}\ar[r]\ar[u] & \dots \\
M\ar[r]^{\id}\ar[u]^{g_{0}} & M\ar[r]^{\id}\ar[u] & M\ar[r]\ar[u] & \dots \\
P_{0}=X_{0}\ar[r]^{l_{0}}\ar[u]^{f_{0}} & X_{1}\ar[r]^{l_{1}}\ar[u] & X_{2}\ar[u]\ar[r] & \dots
}
\]
in which every column is a triangle. The octahedron axiom implies that each $X_{i}$ is an $i$-fold extension of small coproducts of
non negative shifts of objects of $\cs$. Now, by using \cite[Proposition 1.1.11]{BBD} (\ie, Verdier's $3\times 3$ lemma) we get a diagram
\[
\xymatrix{
\coprod_{i\geq 0}M\ar[r]\ar[d]^{\id-\text{shift}} & \coprod_{i\geq 0}Y_{i}\ar[r]\ar[d]^{\id-\text{shift}} & \coprod_{i\geq 0}\Sigma X_{i}\ar[r]\ar[d] & \Sigma\coprod_{i\geq 0}M\ar[d] \\
\coprod_{i\geq 0}M\ar[r]\ar[d] & \coprod_{i\geq 0}Y_{i}\ar[r]\ar[d] & \coprod_{n\geq 0}\Sigma X_{i}\ar[r]\ar[d] & \Sigma\coprod_{i\geq 0}M\ar[d] \\
M\ar[r]\ar[d] & \Mcolim Y_{i}\ar[r]\ar[d] & X'\ar[r]\ar[d] & \Sigma M\ar[d] \\
\Sigma\coprod_{i\geq 0}M\ar[r] & \Sigma\coprod_{i\geq 0}Y_{i}\ar[r] & \Sigma\coprod_{i\geq 0}\Sigma X_{i}\ar[r] & \Sigma^2\coprod_{i\geq 0}M \\
}
\]
where the columns and rows are triangles. It is clear that $\Sigma^{-1}X'\in\op{Susp}_{\cd}(\cs)$. On the other hand, for each $S\in\cs$ and each $n\geq 0$ we have
\[\Hom_{\cd}(\Sigma^nS,\op{Mcolim}Y_{i})\cong\op{colim}_{i\in\N}\Hom_{\cd}(\Sigma^nS,Y_{i})=0
\]
because the induced morphisms
\[\Hom_{\cd}(\Sigma^nS,Y_{i})\ra \Hom_{\cd}(\Sigma^nS,Y_{i+1})
\]
vanish.
\end{proof}

Of course, one would like to express the objects of the left aisle of $t_{\cs}$ in terms of the objects of $\cs$, for instance as a kind of colimit. In \cite[Proposition III.2.6]{BeligiannisReiten} it is proved that this is the case when $\cs$ satisfies a certain vanishing condition. Here we give an alternative proof of this result:


\begin{theorem}\label{from compact objects to aisles}
Let $\cd$ be a triangulated category with small coproducts, and let $\cs$ be a set of compact objects in $\cd$ such that
\[\Hom_{\cd}(S,\Sigma^n S')=0
\]
for all $S, S'\in\cs$ and each $n\geq 1$. Then every object of $\op{Susp}_{\cd}(\cs)$ is the Milnor colimit of a sequence $X_{0}\ra X_{1}\ra X_{2}\ra\dots$ where $X_{i}$ is an $i$-fold extension of small coproducts of non negative shifts of objects of $\cs$.
\end{theorem}
\begin{proof}
Given $M\in\cd$ we will inductively construct a commutative diagram
\[\xymatrix{X_{0}\ar[r]^{f_{0}}\ar[dr]_{\pi_{0}} & X_{1}\ar[r]^{f_{1}}\ar[d]^{\pi_{1}} & \dots\ar[r] & X_{q}\ar[r]^{f_{q}\hspace{0.7cm}}\ar[dll]^{\pi_{q}} & \dots\ko q\geq 0 \\
& M &&&
}
\]
such that:
\begin{enumerate}[a)]
\item $X_{i}$ is an $i$-fold extension of small coproducts of non negative shifts of objects of $\cs$,
\item $\pi_{i}$ induces a surjection
\[\pi_{i}^{\we}:\Hom_{\cd}(\Sigma^nS,X_{i})\ra\Hom_{\cd}(\Sigma^nS,M)
\]
for each $S\in\cs\ko n\geq 0$.
\end{enumerate}
For $i=0$ we take $X_{0}=\coprod_{S\in\cs}\coprod_{n\geq 0}\coprod_{\Hom_{\cd}(\Sigma^nS,M)}\Sigma^nS$ and the obvious morphism
\[\pi_{0}: X_{0}\ra M.
\]
Suppose for some $i\geq 0$ we have constructed $X_{i}$ and $\pi_{i}$. Consider the triangle
\[C_{i}\arr{\alpha_{i}}X_{i}\arr{\pi_{i}}M\ra \Sigma C_{i}
\]
induced by $\pi_{i}$. Consider $Z_{i}=\coprod_{S\in\cs}\coprod_{n\geq 0}\coprod_{\Hom_{\cd}(\Sigma^nS,C_{i})}\Sigma^nS$ and the obvious morphism
\[\beta_{i}:Z_{i}\ra C_{i}.
\]
The triangle
\[Z_{i}\arr{\alpha_{i}\beta_{i}}X_{i}\ra X_{i+1}\ra \Sigma Z_{i}
\]
defines $X_{i+1}$ up to non unique isomorphism. Note that the surjectivity required for $\pi_{i+1}^{\we}$ comes from the surjectivity of $\pi_{i}^{\we}$.

Define $X_{\infty}$ to be the Milnor colimit of the sequence $f_{i}\ko i\geq 0$:
\[
\xymatrix{
\coprod_{i\geq 0}X_{i}\ar[r]^{\id-\sigma} & \coprod_{i\geq 0}X_{i}\ar[r]^{\psi} & X_{\infty}\ar[r] & \Sigma\coprod_{i\geq 0}X_{i}.
}
\]
Consider the morphism
\[\theta=\left[\begin{array}{ccc}\pi_{0}&\pi_{1}&\dots\end{array}\right]:\coprod_{i\geq 0}X_{i}\ra M.
\]
Since $\pi_{i+1}f_{i}=\pi_{i}$ for every $i\geq 0$, we have $\theta(\id-\sigma)=0$, and so we obtain a morphism $\pi_{\infty}:X_{\infty}\ra M$ such that $\pi_{\infty}\psi=\theta$. If we prove that $\pi_{\infty}$ induces an isomorphism
\[
\pi^{\we}_{\infty}:\Hom_{\cd}(\Sigma^nS,X_{\infty})\arr{\sim}\Hom_{\cd}(\Sigma^nS,M)
\]
for every $S\in\cs\ko n\geq 0$, then we have
\[\Hom_{\cd}(\Sigma^nS,\cone(\pi_{\infty}))=0
\]
for every $S\in\cs\ko n\geq 1$. For the case $n=0$, let us consider the exact sequence
\[\Hom_{\cd}(S,X_{\infty})\arr{\sim}\Hom_{\cd}(S,M)\ra\Hom_{\cd}(S,\cone(\pi_{\infty}))\ra\Hom_{\cd}(S,\Sigma X_{\infty})
\]
Since $S$ is compact, there exists a short exact sequence
\[
\coprod_{i\geq 0}\Hom_{\cd}(S,\Sigma X_{i})\ra\Hom_{\cd}(S,\Sigma X_{\infty})\ra\coprod_{i\geq 0}\Hom_{\cd}(S,\Sigma^2X_{i})
\]
>From the hypothesis on the set $\cs$ and the construction of the objects $X_{i}$ we can deduce that both the left and the right hand side of the former sequence vanish, and so $\Hom_{\cd}(S,\Sigma X_{\infty})=0$. Therefore, then we would have
\[
\Hom_{\cd}(\Sigma^nS,\cone(\pi_{\infty}))=0
\]
for every $S\in\cs\ko n\geq 0$. This, by infinite d\'evissage, implies that
\[\Hom_{\cd}(N,\cone(\pi_{\infty}))=0
\]
for each $N\in\op{Susp}(\cs)$. Hence, we have proved that $\op{Susp}_{\cd}(\cs)$ is an aisle in $\cd$.

Let us prove the required bijectivity for $\pi^{\we}_{\infty}$. The surjectivity follows from the identity $\pi^{\we}_{\infty}\psi^{\we}=\theta^{\we}$ and the fact that $\theta^{\we}$ is surjective (thanks to the surjectivity of the $\pi_{i}^{\we}\ko i\geq 0$ and the compactness of the $S\in\cs$). Now consider the commutative diagram
\[
\xymatrix{\coprod_{i\geq 0}\Hom_{\cd}(\Sigma^nS,X_{i})\ar[r]^{(\id-\sigma)^{\we}} & \coprod_{i\geq 0}\Hom_{\cd}(\Sigma^nS,X_{i})\ar[r]^{\psi^{\we}}\ar[dr]_{\theta^{\we}} & \Hom_{\cd}(\Sigma^nS,X_{\infty})\ar[r]\ar[d]^{\pi^{\we}_{\infty}} & 0 \\
&& \Hom_{\cd}(\Sigma^nS,M) &
}
\]
The map $\psi^{\we}$ is surjective since the map
\[
(\Sigma(\id-\sigma))^{\we}:\coprod_{i\geq 0}\Hom_{\cd}(\Sigma^nS,\Sigma X_{i})\ra \coprod_{i\geq 0}\Hom_{\cd}(\Sigma^nS,\Sigma X_{i})
\]
is injective. If we prove that the kernel of $\theta^{\we}$ is contained in the image of $(\id-\sigma)^{\we}$, then we obtain the injectivity of $\pi^{\we}_{\infty}$ by an easy diagram chase. Let
\[g=\left[\begin{array}{ccc}g_{0} & g_{1} & \dots\end{array}\right]^{t}:\Sigma^nS\ra\coprod_{i\geq 0}X_{i}
\]
be such that
\[\left[\begin{array}{ccc}\pi_{0} & \pi_{1} & \dots\end{array}\right]\left[\begin{array}{ccc}g_{0}&g_{1}&\dots\end{array}\right]^{t}=\pi_{0}g_{0}+\pi_{1}g_{1}+\dots=0.
\]
Notice that there exists an $s\geq 0$ such that $g_{s+1}=g_{s+2}=\dots=0$. Then
\[\pi_{0}g_{0}+\dots +\pi_{s}g_{s}=0
\]
implies
\[\pi_{s}(f_{s-1}\dots f_{0}g_{0}+ f_{s-1}\dots f_{1}g_{1}+\dots +g_{s})=0
\]
and so the morphism
\[f_{s-1}\dots f_{0}g_{0}+ f_{s-1}\dots f_{1}g_{1}+\dots +g_{s}
\]
factors through $\alpha_{s}$:
\[f_{s-1}\dots f_{0}g_{0}+ f_{s-1}\dots f_{1}g_{1}+\dots +g_{s}=\alpha_{s}\gamma_{s}:\Sigma^nS\ra C_{s}\ra X_{s}.
\]
By construction of $Z_{s}$ we have that $\gamma_{s}$ factors through $\beta_{s}$, and so
\[f_{s-1}\dots f_{0}g_{0}+ f_{s-1}\dots f_{1}g_{1}+\dots +g_{s}=\alpha_{s}\beta_{s}\xi_{s}.
\]
This implies
\[f_{s}\dots f_{0}g_{0}+ f_{s}\dots f_{1}g_{1}+\dots +f_{s}g_{s}=f_{s}\alpha_{s}\beta_{s}\xi_{s}=0,
\]
since $f_{s}\alpha_{s}\beta_{s}=0$ by construction of $f_{s}$. Therefore, the morphism
\[h:\Sigma^nS\ra\coprod_{i\geq 0}X_{i}
\]
with non-vanishing components
\[\Sigma^nS\ra X_{r}\ra\coprod_{i\geq 0}X_{i}
\]
induced by
\[g_{r}+\dots + f_{r-1}\dots f_{1}g_{1}+f_{r-1}\dots f_{0}g_{0}: \Sigma^nS\ra X_{r}
\]
with $0\leq r\leq s$, satisfies $\varphi^{\we}(h)=g$.
\end{proof}

In practice, every triangulated category is at the basis of a triangulated derivator (see \cite{Cisinski03}). If we assume that our triangulated category $\cd$  satisfies this property, we can use Appendix 1 to get rid of the extra hypothesis on the set $\cs$ of compact objects, to simplify the proof of Theorem~\ref{from compact objects to aisles} and to enhance the proof of Theorem~\ref{AlonsoJeremiasSouto}.

\begin{theorem}\label{from compact objects to aisles with derivators}
Let $\mathbb{D}$ be a triangulated derivator, and let $\cs$ be a set of compact objects of $\mathbb{D}(e)$. Then:
\begin{itemize}
\item[1)] the smallest full subcategory $\op{Susp}_{\mathbb{D}(e)}(\cs)$ of $\mathbb{D}(e)$ containing $\cs$ and closed under extensions, positive shifts and small coproducts is a left aisle,
\item[2)] every object of $\op{Susp}_{\mathbb{D}(e)}(\cs)$ is the Milnor colimit of a sequence $X_{0}\ra X_{1}\ra X_{2}\ra\dots$ where $X_{i}$ is an $i$-fold extension of small coproducts of non negative shifts of objects of $\cs$.
\end{itemize}
\end{theorem}
\begin{proof}
The proof starts as the one of Theorem\ref{AlonsoJeremiasSouto}. Thus, starting from an object $M$ of $\mathbb{D}(e)$ we produce a diagram of the form
\[
\xymatrix{\Sigma X_{0}\ar[r]^{\Sigma x_{0}} & \Sigma X_{1}\ar[r]^{\Sigma x_{1}} & \Sigma X_{2}\ar[r] & \dots \\
Y_{0}\ar[r]^{y_{0}}\ar[u] & Y_{1}\ar[r]^{y_{1}}\ar[u] & Y_{2}\ar[r]\ar[u] & \dots \\
M\ar[r]^{\id}\ar[u]^{g_{0}} & M\ar[r]^{\id}\ar[u] & M\ar[r]\ar[u] & \dots \\
P_{0}=X_{0}\ar[r]^{l_{0}}\ar[u]^{f_{0}} & X_{1}\ar[r]^{l_{1}}\ar[u] & X_{2}\ar[u]\ar[r] & \dots
}
\]
in which every column is a triangle, each $X_{i}$ is an $i$-fold extension of small coproducts of non negative shifts of objects of $\cs$ and
\[\Hom_{\mathbb{D}(e)}(\Sigma^nS,\op{Mcolim}Y_{i})=0
\]
for each $S\in\cs$ and each $n\geq 0$. Let us regard the rows of this diagram as objects $X\ko M$ and $Y$ of the category $\ul{\Hom}(\N,\mathbb{D}(e))$ of presheaves. Thanks to Corollary
\ref{Milnor colimits are quasi exact} we know that there exists a triangle
\[
\op{Mcolim}X'\ra M\ra \op{Mcolim}Y\ra \Sigma\op{Mcolim}X',
\]
where $X'\in\ul{\Hom}(\N,\mathbb{D}(e))$ is such that $X'_{i}\cong X_{i}$ for each $i\geq 0$. In particular, $X'_{i}\in\op{Susp}_{\mathbb{D}(e)}(\cs)$ for all $i\geq 0$, which implies that $\op{Mcolim}X'\in\op{Susp}_{\mathbb{D}(e)}(\cs)$.
\end{proof}



\bibliographystyle{plain}
\bibliography{mybibliography}

\begin{thebibliography}{10}

\bibitem{AlonsoJeremiasSouto2003}
L.~Alonso~Tarr\'io, A.~Jerem\'ias~L\'opez, and M.~J. Souto~Salorio.
\newblock {Construction of $t$-structures and equivalences of derived
  categories}.
\newblock {\em Trans. Amer. Math. Soc.}, 355(6):2523--2543, 2003.

\bibitem{Amiot}
C.~Amiot.
\newblock {Cluster algebras for algebras of global dimension 2 and quivers with
  potential}.
\newblock {\em Ann. Inst. Fourier}, 59:2525--2590, 2009.

\bibitem{BBD}
A.~A. Beilinson, J.~Bernstein, and P.~Deligne.
\newblock {\em Faisceaux Pervers}, volume 100.
\newblock Ast\'{e}risque, 1982.

\bibitem{BeligiannisReiten}
A.~Beligiannis and I.~Reiten.
\newblock {\em Homological Aspects of Torsion Theories}, volume 188.
\newblock Mem. Am. Math. Soc., 2007.

\bibitem{BensonKrauseSchwede04}
D.~Benson, H.~Krause, and S.~Schwede.
\newblock Realizability of modules over {T}ate cohomology.
\newblock {\em Trans. Amer. Math. Soc.}, 356(9):3621--3668 (electronic), 2004.

\bibitem{BokstedtNeeman1993}
M.~B\"{o}kstedt and A.~Neeman.
\newblock {Homotopy limits in triangulated categories}.
\newblock {\em Compositio Mathematica}, 86(2):209--234, 1993.

\bibitem{Bondarko0704.4003v8}
M.~V. Bondarko.
\newblock {Weight structures vs. $t$-structures; weight filtrations, spectral
  sequences, and complexes (for motives and in general)}.
\newblock arXiv:0704.4003v8 [math.KT].

\bibitem{Bridgeland2005}
T.~Bridgeland.
\newblock {$T$-structures in some local Calabi-Yau varietes}.
\newblock {\em J. Algebra}, 289:453--483, 2005.

\bibitem{CartanEilenberg}
H.~Cartan and S.~Eilenberg.
\newblock {\em Homological Algebra}.
\newblock Princeton University Press, 1965.

\bibitem{Cisinski03}
D.-C. Cisinski.
\newblock Images directes cohomologiques dan les cat\'{e}gories de mod\`{e}les.
\newblock {\em Annales Math\'{e}matiques Blaise Pascal}, 10:195--244, 2003.

\bibitem{CisinskiNeeman2005}
D.-C. Cisinski and A.~Neeman.
\newblock {Additivity for derivator K-theory}.
\newblock {\em Advances in Mathematics}, 217(4):1381--1475, 2008.

\bibitem{DwyerGreenlees02}
W.~G. Dwyer and J.~P.~C. Greenlees.
\newblock Complete modules and torsion modules.
\newblock {\em Amer. J. Math.}, 124(1):199--220, 2002.

\bibitem{DwyerGreenleesIyengar06}
W.~G. Dwyer, J.~P.~C. Greenlees, and S.~Iyengar.
\newblock Duality in algebra and topology.
\newblock {\em Adv. Math.}, 200(2):357--402, 2006.

\bibitem{Ginzburg}
V.~Ginzburg.
\newblock {Calabi-Yau algebras}.
\newblock arXiv:math/0612139v3 [math.AG].

\bibitem{HappelSeminaireAlgebre}
D.~Happel.
\newblock Hochschild cohomology of finite-dimensional algebras.
\newblock In {\em S\'{e}minaire d'Alg\`ebre Paul Dubreil et Marie-Paul
  Malliavin, 39\`eme Anne (Paris, 1987/1988)}, volume 1404 of {\em Lecture
  Notes in Math.}, pages 108--126. Springer, Berlin, 1989.

\bibitem{Iyengar10}
S.~Iyengar.
\newblock Private communication, October 2010.

\bibitem{KellerDeformedCalabiYauCompletions}
B.~Keller.
\newblock {Deformed Calabi-Yau Completions}.
\newblock arXiv:0908.3499v5 [math.RT].

\bibitem{Keller1994a}
B.~Keller.
\newblock {Deriving DG categories}.
\newblock {\em Ann. Scient. Ec. Norm. Sup.}, 27(1):63--102, 1994.

\bibitem{KellerDCU}
B.~Keller.
\newblock Derived categories and their uses.
\newblock In {\em Handbook of algebra}, volume~1. Elsevier, 1996.

\bibitem{KellerNicolas11}
B.~Keller and P.~Nicol\'as.
\newblock in preparation.

\bibitem{KellerVossieck1988b}
B.~Keller and D.~Vossieck.
\newblock {Aisles in derived categories}.
\newblock {\em C. R. Acad. Sci. Paris}, 40(2):239--253, 1988.

\bibitem{KellerYang}
B.~Keller and D.~Yang.
\newblock {Derived equivalences from mutations of quivers with potential}.
\newblock arXiv:0906.0761v3 [math.RT].

\bibitem{KoenigLiu}
S.~Koenig and Y.~Liu.
\newblock Simple-minded systems in stable module categories.
\newblock arXiv:1009.1427v1 [math.RT].

\bibitem{Maltsiniotis2001}
G.~Maltsiniotis.
\newblock {Introduction \`{a} la th\'{e}orie des d\'{e}rivateurs (d'apr\`{e}s
  Grothendieck)}.
\newblock preprint (2001) available at Georges Maltsiniotis' homepage.

\bibitem{Maltsiniotis2007}
G.~Maltsiniotis.
\newblock {La K-th\'{e}orie d'un d\'{e}rivateur triangul\'{e} (suivi d'un
  appendice par B. Keller)}.
\newblock {\em Contemp. Math.}, 431:341--368, 2007.

\bibitem{Milnor1962}
J.~Milnor.
\newblock {On Axiomatic Homology Theory}.
\newblock {\em Pacific J. Math.}, 12:337--341, 1962.

\bibitem{Neeman2001}
A.~Neeman.
\newblock {\em Triangulated categories}, volume 148.
\newblock Princeton University Press, 2001.

\bibitem{Pauksztello10}
D.~Pauksztello.
\newblock A note on compactly generated co-$t$-structures.
\newblock arXiv:1006.5347v2 [math.CT].

\bibitem{Pauksztello08}
D.~Pauksztello.
\newblock Compact corigid objects in triangulated categories and
  co-{$t$}-structures.
\newblock {\em Cent. Eur. J. Math.}, 6(1):25--42, 2008.

\bibitem{RickardRouquier10}
J.~Rickard and R.~Rouquier.
\newblock Stable categories and reconstruction.
\newblock arXiv:1008.1976 [math.RT].

\bibitem{SchnuererThesis}
O.~M. Schn\"{u}rer.
\newblock {Equivariant Sheaves on Flag Varietes, DG Modules and Formality}.
\newblock Doktorarbeit, Universit\"{a}t Freiburg, 2007.

\bibitem{Schnuerer08}
O.~M. Schn\"{u}rer.
\newblock {Perfect derived categories of positively graded DG algebras}.
\newblock to appear in Applied Categorical Structures, arXiv:0809.4782v2
  [math.RT].

\bibitem{ToenVaquie}
B.~To\"{e}n and M.~Vaqui\'e.
\newblock {Moduli of objects in dg-categories}.
\newblock {\em Ann. Sci. de l'ENS}, 40(3):387--444, 2007.

\end{thebibliography}

\printindex
\end{document}